\documentclass{amsart}

\usepackage{amssymb,color}
\usepackage{amsfonts}
\usepackage{amsmath}
\usepackage{euscript}
\usepackage{enumerate}
\usepackage{pdfsync}
\newtheorem{theorem}{Theorem}[section]
\newtheorem{lemma}[theorem]{Lemma}

\newtheorem{prop}[theorem]{Proposition}

\newtheorem*{Theorem1'}{Theorem 1'}

\theoremstyle{definition}

\theoremstyle{remark}

\newcommand \Z{{\mathbb Z}}

\newcommand \N{{\mathbb N}}

\begin{document}

\title[On the isomorphism problem for certain $p$-groups]{On the isomorphism problem for certain $p$-groups}

\author{Alexander Montoya Ocampo}
\address{Department of Mathematics and Statistics, University of Regina, Canada}
\email{alexandermontoya1996@gmail.com}

\author{Fernando Szechtman}
\address{Department of Mathematics and Statistics, University of Regina, Canada}
\email{fernando.szechtman@gmail.com}
\thanks{The second author was partially supported by NSERC grant RGPIN-2020-04062}



\begin{abstract} We consider 9 infinite families of finite $p$-groups, for $p$ a prime, and we settle
the isomorphism problem that arises when the parameters that define these groups are modified.
\end{abstract}

\maketitle

\section{Introduction}

We fix $p,m\in\N$ and $\ell\in\Z$, with $p$ is prime and $p\nmid\ell$, set $\alpha=1+p^m\ell$, and consider the
following~9 infinite families of groups. In Case 1, which occurs when $p$ is odd and if $(p,m)=(3,1)$ then $\ell\equiv 1\mod 3$,
we set
\begin{equation}\label{j1}
J_1(\alpha)=\langle x,y\,|\, x^{[x,y]}=x^{1+p^m\ell},\, y^{[y,x]}=y^{1+p^m\ell}, x^{p^{3m}}=y^{p^{3m}}=1\rangle;
\end{equation}
in Case 2, which takes place when $p=2$, we let
\begin{equation}\label{j2}
J_2(\alpha)=\langle x,y\,|\, x^{[x,y]}=x^{1+2^m\ell},\, y^{[y,x]}=y^{1+2^m\ell}, x^{2^{3m-1}}=y^{2^{3m-1}}=1\rangle;
\end{equation}
and in Case 3, which happens when $(p,m)=(3,1)$ and $\ell\equiv -1\mod 3$ we define
\begin{equation}\label{j3}
J_3(\alpha)=\langle x,y\,|\, x^{[x,y]}=x^{1+3\ell},\, y^{[y,x]}=y^{1+3\ell}, x^{81}=y^{81}=1\rangle.
\end{equation}
Moreover, we let $H_i(\alpha)=J_i(\alpha)/Z(J_i(\alpha))$ and $K_i(\alpha)=J_i(\alpha)/Z_2(J_i(\alpha))$
stand for the quotients of $J_i(\alpha)$ by its first and second center, respectively, for $1\leq i\leq 3$.

Given an integer $\beta$, we recall from \cite{M} the 1-parameter Macdonald group:
$$
G(\beta)=\langle x,y\,|\, x^{[x,y]}=x^{\beta},\, y^{[y,x]}=y^{\beta}\rangle.
$$
It is readily seen that $G(1)$ is the integral Heisenberg group. If $\beta\neq 1$ then
$G(\beta)$ is finite \cite{M}, the prime factors of its order being those of $\beta-1$.
Moreover, $G(\beta)$ is also nilpotent. The proof given 
in \cite{M} has a gap that is filled in \cite{MS}.

According to \cite[Theorem 5.3]{MS}, in Case 1, $J_1(\alpha)$ is the Sylow $p$-subgroup of $G(\alpha)$; in Case 2,
$J_2(\alpha)$ is the Sylow $2$-subgroup of $G(\alpha)$; and in Case 3,
$J_3(\alpha)$ is the Sylow $3$-subgroup of $G(\alpha)$. Moreover, \cite[Theorem 7.1]{MS} proves that
$J_1(\alpha)$ has order~$p^{7m}$ and nilpotency class 5;
$J_2(\alpha)$ has order~$2^{7m-3}$, and nilpotency class 5 if $m>1$ and 3 if $m=1$; and
$J_3(\alpha)$ has order $3^{10}$ and nilpotency class 7.

Let $\ell'$ be an integer henceforth assumed to satisfy the same conditions imposed on $\ell$, and
set $\alpha'=1+p^m\ell'$. As hinted above and seen more thoroughly in \cite{MS} and below, the groups
\begin{equation}\label{9g}
J_1(\alpha), H_1(\alpha), K_1(\alpha),J_2(\alpha), H_2(\alpha),K_2(\alpha),J_3(\alpha),H_3(\alpha),K_3(\alpha)
\end{equation}
possess several structural properties that are completely independent of $\ell$, and the question arises
as to whether the groups \eqref{9g} are actually isomorphic to the corresponding counterparts
when $\alpha$ is replaced by $\alpha'$ and, if not, when does such isomorphism exists.

In this paper, we find necessary and sufficient conditions for the existence
of an isomorphism arising from the substitution $\alpha\leftrightarrow\alpha'$ in all 9 cases appearing in (\ref{9g}).

There are infinitely many substitutions $\alpha\leftrightarrow\alpha'$ to consider in each case listed in (\ref{9g}),
although a quick observation leaves only finitely
many to be analyzed. Indeed, it is easy to see that:
if $\alpha^\prime\equiv\alpha\mod p^{3m}$ then $J_1(\alpha^\prime)\cong J_1(\alpha)$; if
$\alpha^\prime\equiv\alpha\mod 2^{3m-1}$ then $J_2(\alpha^\prime)\cong J_2(\alpha)$; if
$\alpha^\prime\equiv\alpha\mod 81$ then $J_3(\alpha^\prime)\cong J_3(\alpha)$. The same observations
are valid for $H_i(\alpha)$ and $K_i(\alpha)$, the moduli being $p^{2m}, 2^{2m-1}$, and 27 in this case.
However, these evident sufficient conditions are too strong,
as the given isomorphisms actually hold under weaker conditions, as indicated below.



\medskip

\noindent{\bf Theorem A. }{\it In Case 1, $J_1(\alpha)\cong J_1(\alpha')$
if and only if $\alpha\equiv \alpha'\mod p^{2m}$. In particular, for fixed $p$ and $m$, there are exactly $\varphi(p^m)$ isomorphism classes of groups $J_1(\alpha)$, except when $(p,m)=(3,1)$, in which case $J_1(1+3\ell)\cong J_1(4)$.
On the other hand, all proper terms and all factors of the upper central series of $J_1(\alpha)$ and $J_1(\alpha')$
are isomorphic. The exact same result is valid
for $H_1(\alpha)$, whereas $K_1(1+p^m\ell)\cong K_1(1+p^m)$.
}

\medskip

\noindent{\bf Theorem B. }{\it In Case 2, $J_2(1+2\ell)\cong J_2(3)$,  $J_2(1+4\ell)\cong J_2(5)$,
and for $m>2$, $J_2(\alpha)\cong J_2(\alpha')$ if
and only if $\alpha\equiv\alpha'\mod 2^{2m}$. In particular, for fixed~$m>2$, there are exactly $\varphi(2^{m})$ isomorphism classes
of groups $J_2(\alpha)$. On the other hand, the derived subgroups as well as
all factors of the upper central series of $J_2(\alpha)$ and $J_2(\alpha')$ are isomorphic. In addition,
$K_2(1+2^m\ell)\cong K_2(1+2^m)$.
}

\medskip

\noindent{\bf Theorem C. }{\it In Case 3, $J_3(1+3\ell)\cong J(7)$, so $H_3(1+3\ell)\cong H_3(7)$ and $K_3(1+3\ell)\cong K_3(7)$.
}

\medskip

\noindent{\bf Theorem D. }{\it Let $\beta$ be any integer. Then the Macdonald groups $G(\alpha)$ and $G(\beta)$
are isomorphic if and only if $\beta=\alpha$, or $\beta=2-\alpha$
and $\alpha\in\{-3,-1,0,1,2,3,5\}$. On the other hand, if $v_3(\alpha-1)\neq 1$, then $G(\alpha)$ and $G(2-\alpha)$
have the same order, and the derived subgroups as well as all factors of the upper central series of $G(\alpha)$ and $G(2-\alpha)$
are isomorphic.
}

\medskip

\noindent{\bf Theorem E. }{\it In Case 2, $H_2(1+2\ell)\cong H_2(3)$, $H_2(1+4\ell)\cong H_2(5)$, $H_2(1+8\ell)\cong H_2(9)$, and for $m>2$,
$H_2(\alpha')\cong H_2(\alpha)$ if and only if $\alpha'\equiv\alpha\mod 2^{2m-2}$. In particular, for fixed~$m>2$, there are
exactly $\varphi(2^{m-2})$ isomorphism classes of groups $H_2(\alpha)$.
}

We found it difficult to distinguish nonisomorphic groups that appear similar in so many respects and we were unable
to detect a structural property to tell them apart. This phenomenon is not unique, see \cite{LW}, for instance.
In our case, we used an isomorphism searching technique, as described in Section \ref{si}, together with our knowledge
of the automorphism group of the groups involved, as given in \cite{MS2} and \cite{MS3}. For an algorithm approach
to the general isomorphism problem for finite $p$-groups, see \cite{O}. As we are dealing with all prime numbers $p$
and all natural numbers $m$ at once, this approach is not applicable to us.

In terms of notation, given a group $T$ and $i\geq 0$, we let $\langle 1\rangle=Z_0(T),Z_1(T),Z_2(T),\dots$ stand for the terms of the upper central series of~$T$, so that $Z_{i+1}(T)/Z_i(T)$ is the center of $T/Z_i(T)$, and we write $\mathrm{Aut}_i(T)$ for the kernel of the canonical map $\mathrm{Aut}(T)\to \mathrm{Aut}(T/Z_i(T))$. We write $[a,b]=a^{-1}b^{-1}ab$ and $b^a=a^{-1}ba$ for $a,b\in T$.
Function composition proceeds from left to right. If $S$ is a normal subgroup of $T$,
we sometimes write $\overline{T}$ for $T/S$ and $\overline{t}$ for $tS\in \overline{T}$.

\section{Sufficiency}

\begin{theorem}\label{suf} (a) In Case 1, assume
$\ell'\equiv \ell\mod p^{m}$, i.e., $\alpha'\equiv\alpha\mod p^{2m}$. We then have
$J_1(\alpha')\cong J_1(\alpha)$.

(b) In Case 2, suppose that $\ell'\equiv \ell\mod 2^{m}$, that is, $\alpha'\equiv\alpha\mod 2^{2m}$. Then
$J_2(\alpha')\cong J_2(\alpha)$.

(c) In Case 3, $J_3(\alpha')\cong J_3(\alpha)$.

\end{theorem}

\begin{proof} (a) We have $\ell'=\ell+rp^m$ for some $r\in\Z$, so that $\alpha'=\alpha+rp^{2m}$.
Let $f=1+kp^m$, with $k\in\N$. Then $\alpha^{kp^m} \equiv 1+k\ell p^{2m}\mod p^{3m}$,
so
$$
\alpha^f \equiv (1+k\ell p^{2m})(1+\ell p^{2m})\equiv \alpha+ k\ell p^{2m}\mod p^{3m}.
$$
Since $p\nmid \ell$, we can select $k$ so that $k\ell\equiv r\mod p^m$, whence $\alpha'\equiv\alpha^f\mod p^{3m}$.

By definition, we have $J_1(\alpha')=\langle X,Y\,|\, X^{[X,Y]}=X^{\alpha'},\, Y^{[Y,X]}=Y^{\alpha'}, X ^{p^{3m}}=Y^{p^{3m}}=1\rangle$,
while $J_1(\alpha)=\langle A,B\,|\, A^{[A,B]}=A^{\alpha},\, B^{[B,A]}=B^{\alpha}, A^{p^{3m}}=B^{p^{3m}}=1\rangle.$
We claim that the assignment
\begin{equation}
\label{iso}
X\mapsto A,\; Y\mapsto B^f,
\end{equation}
extends to an isomorphism $J_1(\alpha^\prime)\to J_1(\alpha)$.
Setting $C=[A,B]$, we must verify that
$$
A^{[A,B^f]}=A^{\alpha'}=A^{\alpha^f}=A^{C^f},\; B^{[B^f,A]}=B^{\alpha'}=B^{\alpha^f}=B^{C^{-f}}.
$$
By \cite[Proposition 7.3]{MS}, $\langle A\rangle$ (resp. $\langle B\rangle$) is equal to its centralizer in $J_1(\alpha)$,
so the above conditions become
$$
[A,B^f]C^{-f}\in \langle A\rangle,\; [B^f,A]C^{f}\in \langle B\rangle.
$$

From $A^C=A^\alpha$ and $B^{C^{-1}}=B^\alpha$ we readily see that
$$
[A,B^f]=B^{(\alpha-1)(\alpha+2\alpha^2+\cdots+(f-1)\alpha^{f-1})} C^f.
$$
Thus $[B^f,A]C^{f}\in \langle B\rangle$ is equivalent to $C^{-f}B^{-(\alpha-1)(\alpha+2\alpha^2+\cdots+(f-1)\alpha^{f-1})} C^{f}\in
\langle B\rangle$, which is automatically true since $C$ normalizes $\langle B\rangle$, while
$[A,B^f]C^{-f}\in \langle A\rangle$ is equivalent to
\begin{equation}
\label{cond}
B^{(\alpha-1)(\alpha+2\alpha^2+\cdots+(f-1)\alpha^{f-1})}\in
\langle A\rangle.
\end{equation}
In this regard,  by \cite[Section 6]{MS}, we have the fundamental relation $A^{p^{2m}}B^{p^{2m}}=1$. In particular,
$B^{p^{2m}}\in \langle A\rangle$. In fact, \cite[Proposition 7.2]{MS} ensures that
$\langle A\rangle\cap \langle B\rangle= \langle B^{p^{2m}} \rangle$.
Thus, (\ref{cond}) is actually equivalent to
\begin{equation}
\label{suma}
\alpha+2\alpha^2+\cdots+(f-1)\alpha^{f-1}\equiv 0\mod p^m.
\end{equation}
Here $\alpha\equiv 1\mod p^m$, so (\ref{suma}) translates into $(f-1)f/2\equiv 1+2+\cdots+(f-1)\equiv 0\mod p^m$,
which is certainly true since $f\equiv 1\mod p^m$.


(2) By definition, we have $J_2(\alpha')=\langle X,Y\,|\, X^{[X,Y]}=X^{\alpha'},\, Y^{[Y,X]}=Y^{\alpha'}, X^{2^{3m-1}}=Y^{2^{3m-1}}=1\rangle$,
while $J_2(\alpha)=\langle A,B\,|\, A^{[A,B]}=A^{\alpha},\, B^{[B,A]}=B^{\alpha}, A^{2^{2m-1}}=B^{p^{2m-1}}=1\rangle.$

We have $\ell'=\ell+r2^m$ for some $r\in\Z$, so that
$\alpha'=\alpha+r2^{2m}$. Let $f=1+k2^m$, with $k\in\N$. Then $\alpha^{k2^m} \equiv 1+k\ell 2^{2m}\mod 2^{3m-1}$, so
$\alpha^f \equiv \alpha+ k\ell 2^{2m}\mod 2^{3m-1}$. As $\ell$ is odd, we can find $k$ such that
$k\ell\equiv r\mod 2^{m-1}$, whence $\alpha'\equiv\alpha^f\mod 2^{3m-1}$.

We claim that the assignment (\ref{iso}) extends to an isomorphism $J_2(\alpha^\prime)\to J_2(\alpha)$.
By the argument from Case 1, we are reduced to verify (\ref{cond}). In this case, by \cite[Section 6]{MS},
we have the fundamental relation $A^{2^{2m-1}}B^{2^{2m-1}}=1$. In particular, $B^{2^{2m-1}}\in\langle A\rangle$.
In fact, \cite[Proposition 7.2]{MS} ensures that
$\langle A\rangle\cap \langle B\rangle= \langle B^{2^{2m-1}} \rangle$. Thus, (\ref{cond})
is actually equivalent to
\begin{equation}
\label{suma2}
\alpha+2\alpha^2+\cdots+(f-1)\alpha^{f-1}\equiv 0\mod 2^{m-1}.
\end{equation}
Here $\alpha\equiv 1\mod 2^m$, so (\ref{suma2}) translates into $(f-1)f/2\equiv 1+2+\cdots+(f-1)\equiv 0\mod 2^{m-1}$,
which is certainly true since $f\equiv 1\mod 2^m$.


(3) By definition, we have $J_3(\alpha')=\langle X,Y\,|\, X^{[X,Y]}=X^{\alpha'},\, Y^{[Y,X]}=Y^{\alpha'}, X^{81}=Y^{81}=1\rangle$,
while $J_3(\alpha)=\langle A,B\,|\, A^{[A,B]}=A^{\alpha},\, B^{[B,A]}=B^{\alpha}, A^{81}=B^{81}=1\rangle.$

We have $\alpha'=1+3\ell'$ and $\alpha=1+3\ell$. Here $\ell'=\ell+3r$ for some $r\in\Z$, so
$\alpha'=\alpha+9r$. Let $f=1+3k$, with $k\in\N$. As $\ell\equiv -1\mod 3$, we infer
$$
\alpha^3\equiv 1+9\ell\mod 81,
$$
which implies $\alpha^{3k}\equiv 1+9k\ell\mod 81$, and therefore
$$
\alpha^{f}\equiv\alpha+9k\ell+27k\ell^2\mod 81.
$$
As $\ell+3\ell^2$ is relatively prime to 3, we can find $k$ so that $k(\ell+3\ell^2)\equiv r\mod 9$,
in which case $\alpha^f\equiv\alpha'\mod 81$.
By the argument from Case 1, we are reduced to verify (\ref{cond}). In this case, by \cite[Section 6]{MS}, we have the fundamental relation
$A^{27}B^{27}=1$.
In particular, $B^{27}\in \langle A\rangle$. In fact, \cite[Proposition 7.2]{MS} ensures that
$\langle A\rangle\cap \langle B\rangle= \langle B^{27} \rangle$. Thus, (\ref{cond})
is actually equivalent to
\begin{equation}
\label{suma3}
\alpha+2\alpha^2+\cdots+(f-1)\alpha^{f-1}\equiv 0\mod 9.
\end{equation}
Here $\alpha\equiv 1\mod 3$, so $\alpha^i\equiv (1+(\alpha-1))^i\equiv 1+i(\alpha-1)\mod 9$, $i\geq 0$,
so (\ref{suma3}) translates into
\begin{equation}
\label{suma4}
(f-1)f/2+(\alpha-1)(1^2+2^2+\cdots+(f-1)^2)\equiv 0\mod 9.
\end{equation}
As is well-known $1^2+2^2+\cdots+n^2=n(n+1)(2n+1)/6$, $n\geq 1$,
so (\ref{suma4}) becomes
\begin{equation}
\label{suma5}
kf[1+(2(f-1)+1)\ell]\equiv 0\mod 3.
\end{equation}
But $\ell\equiv -1\mod 3$, so $1+(2(f-1)+1)\ell\equiv -2(f-1)\equiv 0\mod 3$, as required.
\end{proof}


\section{Searching for isomorphisms}\label{si}

Let $G_1$ and $G_2$ be groups with characteristic subgroups $N_1$ and $N_2$, respectively,
and write $L_1=G_1/N_1$ and $L_2=G_2/N_2$, with canonical projections $\pi_1:G_1\to L_1$
and $\pi_2:G_2\to L_2$.

As $N$ is characteristic subgroup of $G$, we have a homomorphism $\Lambda:\mathrm{Aut}(G_1)\to \mathrm{Aut}(L_1)$
such that for all $\beta\in \mathrm{Aut}(G_1)$, $\beta^{\Lambda}$ is the only automorphism of $L_1$ satisfying
$\pi_1 {\beta}^{\Lambda}=\beta\pi_1$. 

Likewise, to every isomorphism $\Omega:G_1\to G_2$ such that $N_1^\Omega=N_2$ there corresponds a unique
isomorphism $\Omega^*:L_1\to L_2$ satisfying $\pi_1\Omega^*=\Omega\pi_2$. Note that
\begin{equation}\label{match}
(\beta\Omega)^*=\beta^{\Lambda} \Omega^*,\quad \beta\in \mathrm{Aut}(G_1),
\end{equation}
as $\beta^{\Lambda} \Omega^*:L_1\to L_2$  is an isomorphism and $\pi_1\beta^{\Lambda} \Omega^*=\beta\pi_1\Omega^*=\beta\Omega\pi_2$.

We say that an isomorphism $\delta:L_1\to L_2$ lifts if there is an isomorphism $\Omega:G_1\to G_2$ such
that $N_1^\Omega=N_2$ and $\Omega^*=\delta$.

\begin{prop}\label{tool} Let $\gamma:L_1\to L_2$ be an isomorphism and $U$ a subset of $\mathrm{Aut}(L_1)$ such that
$\mathrm{Aut}(L_1)=\mathrm{Aut}(G_1)^{\Lambda}\, U$. Then, there is an isomorphism $\Omega:G_1\to G_2$ such that $N_1^\Omega=N_2$
if and only if there is some $u\in U$ such that $u\gamma$ lifts.
\end{prop}

\begin{proof} Sufficiency is clear. As for necessity, $\Omega^*\gamma^{-1}\in \mathrm{Aut}(L_1)$, so $\Omega^*=\beta^{\Lambda} u\gamma$
for some $\beta\in \mathrm{Aut}(G_1)$ and $u\in U$. It then follows from (\ref{match}) that
$(\beta^{-1}\Omega)^*=u\gamma$.
\end{proof}

\begin{prop}\label{tool2} Let $\gamma:L_1\to L_2$ be an isomorphism, $U$ a subset of $\mathrm{Aut}(L_1)$ such that
$\mathrm{Aut}(L_1)=\mathrm{Aut}(G_1)^\Lambda\, U$, and $E$ a generating subset of $G_1$.
For $e\in E$ and $u\in U$, let $h_{e,u}$ be any element of $G_2$ satisfying $e \pi_1 u \gamma=h_{e,u}\pi_2$.
Suppose that given
any $(f_e)_{e\in E}$, with $f_e\in N_1$, there exists $\beta\in \mathrm{Aut}(G_1)$ such that $e^\beta=e f_e$ for all $e\in E$.

Then, there is an isomorphism $\Omega:G_1\to G_2$ such that $N_1^\Omega=N_2$ if and only if
there is some $u\in U$ such that the assignment
$$
e\mapsto h_{e,u},\quad e\in E,
$$
extends to an isomorphism $\Delta:G_1\to G_2$ satisfying $N_1^\Delta=N_2$.
\end{prop}

\begin{proof} Sufficiency is clear. As for necessity, by Proposition \ref{tool}, there is some $u\in U$
and an isomorphism $\Delta:G_1\to G_2$ such that $N_1^\Delta=N_2$ and $\Delta^*=u\gamma$. Then
$$
h_{e,u}\pi_2=e \pi_1 u\gamma=e\Delta\pi_2,\quad e\in E,
$$
so there exist $(k_e)_{e\in E}$, with $k_e\in N_2$, such that $h_{e,u}k_e=e\Delta$ for all $e\in E$. Since $N_1^\Delta=N_2$,
there exist $(f_e)_{e\in E}$, with $f_e\in N_1$, such that $f_e^\Delta=k_e^{-1}$ for all $e\in E$. By assumption, there
exists $\beta\in \mathrm{Aut}(G_1)$ such that $e^\beta=e f_e$. Then $e^{\beta\Delta}=h_{e,u}$ for all $e\in E$.
\end{proof}

The above will be applied when $N_1$ and $N_2$ are the centers of $G_1$ and $G_2$, respectively,
in which case $N_1$ and $N_2$ are characteristic subgroups of $G_1$ and $G_2$, and any isomorphism
$\Omega:G_1\to G_2$ satisfies $N_1^\Omega=N_2$.

\section{Necessity for $J_1(\alpha)$}\label{caso1}

We assume throughout this section that we are in Case 1.
Note that $\alpha$ and $\alpha'$ have order $p^m$ modulo $p^{2m}$, so that
$\alpha'\equiv\alpha^t\mod p^{2m}$, where $t$ is a positive integer relatively prime to $p$, uniquely
determined modulo $p^{m}$. Thus, the following conditions are equivalent: $\ell\equiv \ell'\mod p^m$;
$\alpha\equiv\alpha'\mod p^{2m}$; $t\equiv 1\mod p^m$.

By \cite[Section 3]{MS2}, we have $H_1(\alpha)=\langle A,B\,|\, A^{[A,B]}=A^\alpha,\, B^{[B,A]}=B^\alpha, A^{p^{2m}}=1=B^{p^{2m}}\rangle$
and
$
H_1(\alpha')=\langle X,Y\,|\, X^{[X,Y]}=X^{\alpha'},\, Y^{[Y,X]}=Y^{\alpha'},  X^{p^{2m}}=1=Y^{p^{2m}}\rangle.
$
Moreover, by \cite[Section 2]{MS2}, we have
$
K_1(\alpha)=\langle a,b\,|\, a^{[a,b]}=a^\alpha,\, b^{[b,a]}=b^\alpha, a^{p^{2m}}=1=b^{p^{2m}}, [a,b]^{p^{m}}=1\rangle
$
and
$
K_1(\alpha')=\langle x,y\,|\, x^{[x,y]}=x^{\alpha'},\, y^{[y,x]}=y^{\alpha'},  x^{p^{2m}}=1=y^{p^{2m}},[x,y]^{p^{m}}=1\rangle.
$
There are automorphisms, say $\nu$ and $\mu$, of $H_1(\alpha')$ and $K_1(\alpha')$, given by $X\leftrightarrow Y$
and $x\leftrightarrow y$.

Consider the projection homomorphisms $\pi_\alpha:H_1(\alpha)\to K_1(\alpha)$, given by $A\mapsto a$ and $B\mapsto b$,
and $\pi_{\alpha'}:H_1(\alpha')\to K_1(\alpha')$, given by $X\mapsto x$ and $Y\mapsto y$.

We have the natural homomorphism $\Lambda:\mathrm{Aut}(H_1(\alpha'))\to \mathrm{Aut}(K_1(\alpha'))$ such that
$\pi_{\alpha'}\beta^\Lambda= \beta\pi_{\alpha'}$ for all $\beta\in \mathrm{Aut}(H_1(\alpha'))$.

\begin{prop}\label{authh} For every $z_1,z_2\in Z(H_1(\alpha'))$, the assignment $X\mapsto Xz_1,\; Y\mapsto Yz_2$
extends to an automorphism of $H_1(\alpha')$.
\end{prop}

\begin{proof} This is shown in \cite[Proposition 3.1]{MS2}.
\end{proof}

\begin{prop}\label{suf2} The assignment
\begin{equation}
\label{cta}
x\mapsto a,\; y\mapsto b^t
\end{equation}
extends to an isomorphism $\gamma:K_1(\alpha')\to K_1(\alpha)$.
\end{prop}

\begin{proof} This is shown in \cite[Proposition 2.6]{MS2}.
\end{proof}

\begin{prop}\label{tet} The homomorphism
$\Lambda:\mathrm{Aut}(H_1(\alpha'))\to \mathrm{Aut}(K_1(\alpha'))$ sends $\langle\nu\rangle$ onto $\langle\mu\rangle$
and $\mathrm{Aut}_3(H_1(\alpha'))$ onto
$\mathrm{Aut}_2(K_1(\alpha'))$.
\end{prop}

\begin{proof} This is shown in \cite[Proposition 3.4]{MS2}.
\end{proof}

Given $n\in\Z$ relatively prime to $p$ with inverse $q\in\Z$ modulo $p^{2m}$, \cite[Proposition 2.3]{MS2} ensures
that the assignment
\begin{equation*}
x\mapsto x^n,\; y\mapsto y^q
\end{equation*}
extends to an automorphism, say $f_n$, of $K_1(\alpha')$, and we write $U$ for the subgroup of $\mathrm{Aut}(K_1(\alpha'))$
consisting of all such automorphisms.

\begin{theorem}\label{autk6} We have  $\mathrm{Aut}(K_1(\alpha'))=\mathrm{Aut}_2(K_1(\alpha'))\langle \mu \rangle U$.
\end{theorem}

\begin{proof} This is shown in \cite[Theorem 2.7]{MS2}.
\end{proof}

\begin{theorem}\label{exel} Suppose that
$H_1(\alpha')\cong H_1(\alpha)$. Then $\alpha'\equiv\alpha\mod p^{2m}$.
\end{theorem}

\begin{proof} By assumption there is an isomorphism $\Omega:H_1(\alpha')\to H_1(\alpha)$. We next apply Proposition~\ref{tool2}
with $G_1=H_1(\alpha')$, $G_2=H_1(\alpha)$, $\gamma$ as in (\ref{cta}), $E=\{X,Y\}$, and $U$ as given in Theorem \ref{autk6}.
This is a valid application thanks to Propositions \ref{authh} and \ref{tet}, and Theorem \ref{autk6}. If $u=f_n$, then
$u\gamma$ is given by
\begin{equation}
\label{cra}
x\mapsto a^n,\; y\mapsto b^{qt},
\end{equation}
so we can take $h_{X,u}=A^n$ and $h_{Y,u}=B^{qt}$. Thus, by Proposition \ref{tool2}, there is some $n\in\Z$ relatively
prime to $p$ such that the assignment $X\mapsto A^n,\; Y\mapsto B^{qt}$ extends to an isomorphism $P:H_1(\alpha')\to H_1(\alpha)$.
Set $C=[A,B]$. As the relations $X^{[X,Y]}=X^{\alpha'},\, Y^{[Y,X]}=Y^{\alpha'}$ must be preserved by~$P$, it follows that
\begin{equation}
\label{pin}
A^{[A^n,B^{qt}]}=A^{\alpha^t}=A^{C^t},\; B^{[B^{qt},A^n]}=B^{\alpha^t}=B^{C^{-t}}.
\end{equation}

Set $i=-n$ and $j=qt$. Thus $i,j$ are relatively prime to $p$. As $A$ and $B$ have finite order, we may assume without loss that $i,j>0$.
From $A^C=A^\alpha$ and $B^{C^{-1}}=B^\alpha$ we see, as before, that
$$
A^{B^j}=A B^{(\alpha-1)(\alpha+2\alpha^2+\cdots+(j-1)\alpha^{j-1})} C^j.
$$
As $\alpha\equiv 1\mod p^m$ and $B^{p^{2m}}=1$, it follows that
$$
A^{B^j}=A B^{(\alpha-1)j(j-1)/2}C^j.
$$
By \cite[Theorem 8.1]{MS}, we have $Z_2(H_1(\alpha))=\langle A^{p^m},  B^{p^m}, C^{p^m}\rangle$. Thus,
$$
(A^i)^{B^j}\equiv (A B^{(\alpha-1)j(j-1)/2} C^j)^i\equiv B^{(\alpha-1)ij(j-1)/2} (A C^j)^i\mod Z(H_1(\alpha)).
$$
This and \cite[Eq. (1.3)]{M} yield
$$
(A^i)^{B^j}\equiv B^{(\alpha-1)ij(j-1)/2} C^{ij} A^{\alpha^{j}(\alpha^{ij}-1)/(\alpha^{j}-1)}\mod Z(H_1(\alpha)).
$$
Here
$$
\alpha^{j}(\alpha^{ij}-1)/(\alpha^{j}-1)\equiv i+(\alpha-1)ij (i+1)/2\mod p^{2m},
$$
so
$$
(A^i)^{B^j}\equiv B^{(\alpha-1)ij(j-1)/2} C^{ij} A^{i+(\alpha-1)ij (i+1)/2}\mod Z(H_1(\alpha)),
$$
and therefore
$$
(A^{-i})^{B^j}\equiv A^{-i-(\alpha-1)ij (i+1)/2} B^{-(\alpha-1)ij(j-1)/2} C^{-ij}\mod Z(H_1(\alpha)),
$$
$$
[A^{-i},B^j]\equiv A^{-(\alpha-1)ij (i+1)/2} B^{-(\alpha-1)ij(j-1)/2} C^{-ij}\mod Z(H_1(\alpha)).
$$
It now follows from (\ref{pin}) that
$$
B^{-(\alpha-1)ij(j-1)/2} C^{-ij}C^{-t}\in C_{H_1(\alpha)}(A),
A^{(\alpha-1)ij(i+1)/2} C^{ij} C^{t}\in C_{H_1(\alpha)}(B).
$$
In this regard, the proof of \cite[Theorem 8.1]{MS} shows the following:
$C_{H_1(\alpha)}(A)=\langle A, C^{p^m}\rangle$ and $C_{H_1(\alpha)}(B)=\langle B, C^{p^m}\rangle$;
$A$ and $B$ have order $p^{2m}$; every element of $H_1(\alpha)$ can be written in one and only one way
as product of elements taken from $\langle A\rangle$, $\langle B\rangle$, and $\langle C\rangle$,
in any fixed order. We deduce that $j\equiv 1\mod p^m$, $i\equiv -1\mod p^m$,
and $t\equiv -ij\mod p^{m}$, whence $t\equiv 1\mod p^{m}$.
\end{proof}

\section{Necessity for $J_2(\alpha)$ Part I}\label{part1}

We assume in this section that $p=2$ and $m>1$, and we set $s=2^{m-1}$, $u=s^2$, and $r=s/2$.

By \cite[Section 2]{MS3},
$
J_2(\alpha)=\langle A,B\,|\, A^{[A,B]}=A^{\alpha},\, B^{[B,A]}=B^{\alpha},  A^{2^{3m-1}}=1=B^{2^{3m-1}}\rangle
$
and
$
J_2(\alpha')=\langle X,Y\,|\, X^{[X,Y]}=X^{\alpha'},\, Y^{[Y,X]}=Y^{\alpha'},  X^{2^{3m-1}}=1=Y^{2^{3m-1}}\rangle.
$
In addition, as seen in \cite[Section 6]{MS3},
$
H_2(\alpha)=\langle A_0,B_0\,|\, A_0^{[A_0,B_0]}=A_0^\alpha,\, B_0^{[B_0,A_0]}=B_0^\alpha, A_0^{2^{2m-1}}=1=B_0^{2^{2m-1}}\rangle,
$
and
$
H_2(\alpha')=\langle X_0,Y_0\,|\, X_0^{[X_0,Y_0]}=X_0^{\alpha'},\, Y_0^{[Y_0,X_0]}=Y_0^{\alpha'},  X_0^{2^{2m-1}}=1=Y_0^{2^{2m-1}}\rangle.
$
There are automorphisms, say $\theta$ and $\nu$, of $J_2(\alpha')$ and $H_2(\alpha')$, given by $X\leftrightarrow Y$
and $X_0\leftrightarrow Y_0$.

Consider the projection homomorphisms $\pi_\alpha:J_2(\alpha)\to H_2(\alpha)$, given by $A\mapsto A_0$ and $B\mapsto B_0$,
and $\pi_{\alpha'}:J_2(\alpha')\to H_2(\alpha')$, given by $X\mapsto X_0$ and $Y\mapsto Y_0$.

We have the natural homomorphism $\Lambda:\mathrm{Aut}(J_2(\alpha'))\to \mathrm{Aut}(H_2(\alpha'))$ such that $\pi_{\alpha'} \beta^\Lambda=
\beta\pi_{\alpha'}$ for all $\beta\in \mathrm{Aut}(J_2(\alpha'))$.

\begin{prop}\label{cenj} The terms of the upper central series of $J_2(\alpha)$ are:
$$
    Z_1(J_2(\alpha))=\langle A^{{2u}}\rangle,\;
    Z_2(J_2(\alpha))=\langle  A^{{2u}},  C^{s}\rangle, Z_3(J_2(\alpha))=\langle A^{2s},  B^{2s}, C^{s}\rangle,
$$
with $Z_3(J_2(\alpha))$ abelian, $Z_4(J_2(\alpha))=\langle A^{s},  B^{s}, C\rangle$, and $Z_5(J_2(\alpha))=J_2(\alpha)$.
\end{prop}

\begin{proof} This is shown in \cite[Theorem 8.1 and Proposition 9.2]{MS}.
\end{proof}

\begin{prop}\label{autjo}
    For every $z_1,z_2\in Z_2(J_2(\alpha'))$ the assignment
    $X\mapsto Xz_1,\; Y\mapsto Yz_2$ extends to an automorphism of $J_2(\alpha')$.
\end{prop}

\begin{proof} This is shown in \cite[Proposition 7.2]{MS3}.
\end{proof}

\begin{prop}\label{auth}
    For every $x,y\in Z_2(H_2(\alpha'))$ the assignment $X_0\mapsto X_0x,\; Y_0\mapsto Y_0y$ extends to an automorphism
		$\Pi_{(x,y)}$ of $H_2(\alpha')$.
\end{prop}

\begin{proof} This is shown in \cite[Proposition 6.6]{MS3}.
\end{proof}

\begin{prop}\label{gamma} The assignment $X_0\mapsto X_0^{1+s},\; Y_0\mapsto Y_0X_0^s$ extends to an automorphism $\Gamma$ of
$H_2(\alpha')$ that belongs to $\mathrm{Aut}_3(H_2(\alpha'))$.
\end{prop}

\begin{proof} This is shown in \cite[Proposition 6.8]{MS3}.
\end{proof}

\begin{lemma}\label{autlin} Let $G$ be a group with a characteristic subgroup $N$ and let
$\mathrm{Aut}_N(G)$ be the kernel of $\mathrm{Aut}(G)\to \mathrm{Aut}(G/N)$. Given a generating subset $T$
of $G$, suppose that $f,g\in \mathrm{Aut}(G)$ satisfy $t^f\equiv t^g\mod N$ for all $t\in T$. Then $f\equiv g\mod \mathrm{Aut}_N(G)$.
\end{lemma}

\begin{proof} Since $N$ is a characteristic subgroup of $G$, from $t^f\equiv t^g\mod N$ for all $t\in T$, we deduce
$t^{fg^{-1}}\equiv t\mod N$ for all $t\in T$, which implies $fg^{-1}\in \mathrm{Aut}_N(G)$.
\end{proof}

Recall the automorphism $\Pi_{(x,y)}$ of $H_2(\alpha')$ defined in Proposition \ref{auth}.

\begin{theorem}\label{indexj} Set $V=\{\Pi_{(x,y)}\,|\, x,y\in \langle X_0^{2s},Y_0^{2s}\rangle\}$ and $U=V\{1,\Gamma\}$.
Then $H_2(\alpha')=J_2(\alpha')^\Lambda\, U$.
\end{theorem}

\begin{proof} Set $J=J_2(\alpha')$ and $H=H_2(\alpha')$. Note that $\theta^\Lambda=\nu$. We claim that
if $m=2$ then
\begin{equation}\label{hj0}
\mathrm{Aut}(H)=\langle\theta\rangle^\Lambda \mathrm{Aut}_3(H),
\end{equation}
while if $m>2$ then
\begin{equation}\label{hj1}
\mathrm{Aut}(H)=\langle\Sigma_1,\theta\rangle^\Lambda\mathrm{Aut}_3(H),
\end{equation}
where $\Sigma_1\in \mathrm{Aut}(J)$ is given by $X\mapsto X^{1+r}Y^r X^{s(sk+r\ell-2)}$, $Y\mapsto Y^{1+r}X^r$,
where $k$ is even if $m>3$ and odd if $m=3$, according to \cite[Proposition 7.6]{MS3}.
Indeed, if $m=2$, we have $\mathrm{Aut}(H)=\langle \nu\rangle\mathrm{Aut}_3(H)$ by \cite[Theorem 6.14]{MS3}, and the claim follows
in this case. Suppose  $m>2$. Then \cite[Theorem 6.14]{MS3} ensures that
$\mathrm{Aut}(H)=\langle \Sigma,\nu\rangle\mathrm{Aut}_3(H)$, where
$\Sigma$ is defined in \cite[Corollary 6.12]{MS3} by $X_0\mapsto X_0^{1+r}Y_0^r$, $Y_0\mapsto Y_0^{1+r}X_0^r$.
Since $X_0^s\in Z_3(H)$ by Proposition \ref{cenj}, it follows from Lemma \ref{autlin} that
$\Sigma_1^\Lambda\equiv\Sigma\mod \mathrm{Aut}_3(H)$, which proves the claim in this case.

Let $\Gamma\in \mathrm{Aut}_3(H)$ be as defined in Proposition \ref{gamma},
and let $\Gamma_1\in \mathrm{Aut}_4(J)$ be as defined in \cite[Proposition 7.5]{MS3} by
$X\mapsto X^{1+s}Y^s$, $Y\mapsto Y^{1+s(2\delta_{2,m}+s-3)} X^s$, where $\delta_{i,j}$
is the Kronecker delta function. By Theorem \cite[Theorem 6.10]{MS3}, we have
$$
\langle \overline{\Gamma}, \overline{\Gamma^\nu}\rangle=\mathrm{Aut}_3(H)/\mathrm{Inn}(H)\mathrm{Aut}_2(H)\cong \Z/2\Z\times \Z/2\Z,
$$
which implies
$$
\mathrm{Aut}_3(H)=\mathrm{Inn}(H)\mathrm{Aut}_2(H)\langle \Gamma, \Gamma^\nu\rangle=
\mathrm{Inn}(H)\mathrm{Aut}_2(H)\{1, \Gamma\Gamma^\nu\}\{1, \Gamma\}.
$$
Set $Z=[X,Y]$. Then $X^Y=XZ$ and $Y^X=YZ^{-1}$, with $Z\in Z_4(J)$ by Proposition \ref{cenj}. Thus,
$\mathrm{Inn}(J)\subset\mathrm{Aut}_4(J)$ and $\mathrm{Inn}(H)\subset\mathrm{Aut}_3(H)$, so,
by \cite[Proposition 6.9]{MS3}, $\Lambda$ induces an imbedding
$$
\overline{\Lambda}: \mathrm{Aut}_4(J)/\mathrm{Inn}(J)\mathrm{Aut}_3(J)\to \mathrm{Aut}_3(H)/\mathrm{Inn}(H)\mathrm{Aut}_2(H).
$$
In this regard, the proof of \cite[Theorem 7.7]{MS3} shows that
$\overline{\mathrm{Aut}_4(J)}=\langle \overline{\Gamma_1}\rangle$ has order 2 and that $\overline{\Gamma_1}^{\overline{\Lambda}}\neq
\overline{\Gamma}$, i.e., $\Gamma_1^\Lambda\not\equiv \Gamma\mod \mathrm{Inn}(H)\mathrm{Aut}_2(H)$. But then
$(\Gamma_1^\theta)^\Lambda\equiv(\Gamma_1^\Lambda)^\nu\not\equiv \Gamma^\nu\mod \mathrm{Inn}(H)\mathrm{Aut}_2(H)$, i.e.,
$\overline{\Gamma_1}^{\overline{\Lambda}}\neq \overline{\Gamma^\nu}$. This forces
$\overline{\Gamma_1}^{\overline{\Lambda}}=\overline{\Gamma\Gamma^\nu}$, i.e.,
$\Gamma_1^\Lambda\equiv \Gamma\Gamma^\nu\mod \mathrm{Inn}(H)\mathrm{Aut}_2(H)$. Therefore,
\begin{equation}\label{hj2}
\mathrm{Aut}_3(H)=\{1,\Gamma_1\}^\Lambda \mathrm{Inn}(H)\mathrm{Aut}_2(H)\{1, \Gamma\}.
\end{equation}

It follows from Proposition \ref{cenj} that $Z_2(H)=\langle X_0^{2s}, Y_0^{2s}, [X_0,Y_0]^s\rangle$ and
$Z_1(H)=\langle [X_0,Y_0]^s\rangle$. We deduce from \cite[Proposition 6.6]{MS3} that
\begin{equation}\label{hj3}
\mathrm{Aut}_2(H)=\mathrm{Aut}_1(H)V.
\end{equation}
Here
\begin{equation}\label{hj4}
\mathrm{Aut}_1(H)=\mathrm{Aut}_2(J)^\Lambda
\end{equation}
by \cite[Proposition 7.2]{MS3}. On other hand, it is clear that
\begin{equation}\label{hj5}
\mathrm{Inn}(H)=\mathrm{Inn}(J)^\Lambda.
\end{equation}
Combining (\ref{hj0})-(\ref{hj5}) proves that $\mathrm{Aut}(H)=\mathrm{Aut}(J)^\Lambda\, U$.
\end{proof}

The fundamental relations $A^{2u}B^{2u}=1$ and
$A^{2us}=B^{2us}=C^{2u}$, shown in \cite[Section 6]{MS}, will be used without comment,
as well as $Z_1(J_2(\alpha))=\langle A^{2u}\rangle$, as indicated in Proposition \ref{cenj}.

\begin{theorem}\label{necj2} Suppose that $\alpha'\equiv\alpha\mod 2^{2m-1}$ but
$\alpha'\not\equiv\alpha\mod 2^{2m}$. Then $J_2(\alpha')\cong J_2(\alpha)$ if and only if $m=2$.
In particular, $J_2(1+4\ell)\cong J(5)$.
\end{theorem}

\begin{proof} As $\alpha'\equiv\alpha\mod 2^{2m-1}$, it is clear that the assignment $X_0\mapsto A_0$ and $Y_0\mapsto B_0$
extends to an isomorphism $\gamma: H_2(\alpha')\to H_2(\alpha)$. We next apply Proposition~\ref{tool2}
with $G_1=J_2(\alpha')$, $G_2=J_2(\alpha)$, $E=\{X,Y\}$, and $U$ as in Theorem \ref{indexj}.
This is a valid application thanks to Proposition \ref{autjo}. By Proposition \ref{cenj},
$A_0^s\in Z_3(H_2(\alpha))$, and $Z_2(H_2(\alpha))=\langle  A_0^{2s},  B_0^{2s}, C_0^{s}\rangle$ is abelian
with $C_0^{s}\in Z(H_2(\alpha))$. It follows that for $x,y\in\langle A_0^{2s},B_0^{2s}\rangle$, we have
$$
A_0\Pi_{(x,y)}\Gamma\equiv A_0^{1+s+2si}B_0^{2sj}\mod Z(H_2(\alpha)),\;
B_0\Pi_{(x,y)}\Gamma\equiv A_0^s B A_0^{2sa}B_0^{2sb}\mod Z(H_2(\alpha))
$$
for some $i,j,a,b\in\Z$. Thus, for $u=\Pi_{(x,y)}\Gamma$, we can take $h_{X,u}=A^{1+s+2si}B^{2sj}z_1$
and $h_{Y,u}=A^s B A^{2sa}B^{2sb}z_2$, where $z_1,z_2\in Z_2(J_2(\alpha))$. By Proposition \ref{tool2},
$J_2(\alpha')\cong J_2(\alpha)$ if and only
the assignment
$$
X\mapsto A^{1+s+2si}B^{2sj}z_1,\;Y\mapsto A^s B A^{2sa}B^{2sb}z_2
$$
extends to an isomorphism $J_2(\alpha')\to J_2(\alpha)$ for some choice of $i,j,a,b$ and $z_1,z_2\in Z_2(J_2(\alpha))$.
In view of Proposition \ref{autjo} and the fact that any isomorphism $J_2(\alpha')\to J_2(\alpha)$ restricts to an isomorphism
between their second centers, we deduce that
$J_2(\alpha')\cong J_2(\alpha)$ if and only
the assignment
\begin{equation}
\label{as}
X\mapsto A^{1+s+2si}B^{2sj},\;Y\mapsto A^s B A^{2sa}B^{2sb}
\end{equation}
extends to an isomorphism $J_2(\alpha')\to J_2(\alpha)$ for some choice of $i,j,a,b$.

Suppose that (\ref{as}) extends to a homomorphism, say $\Omega$, for some $i,j,a,b$.
 According to Proposition \ref{cenj}, the centers of $J_2(\alpha^\prime)$ and $J_2(\alpha)$ are generated by
$X^{2u}$ and $A^{2u}$, respectively. Since $\Omega$ maps $X^{2u}$ into $A^{(1+s)2u}$, with $1+s$ odd, we see that $\Omega$
restricts an isomorphism between the centers of $J_2(\alpha^\prime)$ and $J_2(\alpha)$. As these are nilpotent groups,
it follows that $\Omega$ is injective.
But $J_2(\alpha^\prime)$ and $J_2(\alpha)$ have order $2^{7m-3}$, so $\Omega$ is an isomorphism.

Now (\ref{as}) extends to a homomorphism $J_2(\alpha')\to J_2(\alpha)$ if and only if there exist $i,j,a,b\in\Z$ such that
the following relations hold in $J_2(\alpha)$:
\begin{equation}
\label{poop}
(A^{1+s+2si}B^{2sj})^{[A^{1+s+2si}B^{2sj},A^s BA^{2sa}B^{2sb}]}=(A^{(1+s+2si)}B^{2sj})^{\alpha'},
\end{equation}
\begin{equation}
\label{poop2}
(A^s B A^{2sa}B^{2sb})^{[A^s BA^{2sa}B^{2sb},A^{1+s+2si}B^{2sj}]}=(A^sBA^{2sa}B^{2sb})^{\alpha'}.
\end{equation}

We proceed to compute the left and right hand sides of (\ref{poop}) and (\ref{poop2}), and set $C=[A,B]$ for this purpose.

Applying \cite[Theorem 3.2]{MS3} to $(A^0BC^0)(A^{2sa}B^{2sb}C^0) = A^{expA + \xi} B^{expB} C^{expC}$ gives
    $expA\equiv 2sa-2u\ell a\mod 4us$,
    $expB\equiv 1+2sb\mod 4us$,
    $expC\equiv -2sa+2ua\mod 4u$, and
    $\xi\equiv 2usa\mod 4us$. Thus
		\begin{equation}
\label{poopicat}
		A^s BA^{2sa}B^{2sb}=A^{s+2sa}B^{1+2sb} C^{-2sa+2ua} A^{2usa-2u\ell a}=A^{s+2sa}B^{1+2sb} C^{-2sa}A^{-2u\ell a}.
		\end{equation}
		Therefore
		$$
		A^s BA^{2sa}B^{2sb}\equiv A^{s+2sa}B^{2sb} C^{-2sa}\mod Z(J_2(\alpha)),
		$$
	  \[
    [A^{1+s+2si}B^{2sj},A^s BA^{2sa}B^{2sb}]
    = [A^{1+s+2si}B^{2sj},A^{s+2sa}B^{2sb}C^{-2sa}].
    \]

    Applying \cite[Proposition 3.4]{MS3} to $[A^{1+s+2si}B^{2sj},A^{s+2sa}BC^{-2sa}] = A^{expA} B^{expB} C^{expC}$, we obtain
    $expA\equiv u\mod 2u$,
    $expB\equiv 0\mod 2u$,
    $expC\equiv 1 + s + 2s(i+b)\mod 2u$, so
    \begin{equation}
    \label{corch}
    [A^{1+s+2si}B^{2sj},A^{s+2sa}B^{2sb}C^{-2sa}]\equiv A^{u} C^{1 + s + 2s(i+b)}\mod Z(J_2(\alpha)).
    \end{equation}
    By Proposition \ref{cenj}, $Z_3(J_2(\alpha))=\langle A^{2s},  B^{2s}, C^{2^{s}}\rangle$ is abelian, so
		$$
		(A^{1+s+2si}B^{2sj})^{[A^{1+s+2si}B^{2sj},A^s BA^{2sa}B^{2sb}]}=
		(A^{1+s+2si}B^{2sj})^{A^{u} C^{1 + s + 2s(i+b)}}=(A^{1+s+2si}B^{2sj})^{C^{1 + s + 2s(i+b)}}.
		$$
		By means of \cite[Theorem 3.1]{MS3}, we may now derive
		\begin{equation}
    \label{poopl}
(A^{1+s+2si}B^{2sj})^{[A^{1+s+2si}B^{2sj},A^s BA^{2sa}B^{2sb}]}
    = A^{\alpha + s + 2si  + 4u\ell(2i + b+j+ 1)}.
\end{equation}

 On the other hand, applying \cite[Theorem 3.5]{MS3} to $(A^{1+s+2si}B^{2sj})^{\alpha^\prime} = A^{expA + \xi} B^{expB} C^{expC}$ gives
    $expA\equiv 1 + 2s\ell' + s + 2si + 2u\ell' + 4u\ell' i\mod 4us$,
    $expB\equiv 2sj+4u\ell'j+2usj\mod 4us$,
    $expC\equiv 2sj\mod 4u$,
    $\xi\equiv 0\mod 4us$, so
   \begin{equation}
    \label{poopr}
(A^{1+s+2si}B^{2sj})^{\alpha^\prime}
    = A^{\alpha' + s + 2si + 2u\ell' + 4u\ell' i}B^{2sj+4u\ell'j+2usj}C^{2sj}=
		A^{\alpha' + s + 2si + 2u\ell' + 4u\ell'(i-j)}B^{2sj}.
   \end{equation}

	Thus (\ref{poop}) holds if and only if (\ref{poopl}) equals (\ref{poopr}). The order $A$ is $2^{3m-1}$ by \cite[Theorem 7.2]{MS}.
	Thus (\ref{poop}) is true if and only if
		\begin{equation}
    \label{poop5}
		2s\ell+4u\ell(2i+b+j+1)\equiv 2s\ell'+2u\ell'+4u\ell'(i-j)\mod 4us.
		\end{equation}
		By  hypothesis $\ell'=\ell+q s$ with $q$ odd, so $q+\ell$ is even. Hence, setting $r=s/2$, (\ref{poop5}) becomes
		$$
		4u\ell(i+b+j+1)\equiv 4u(r+(q+\ell)/2) \mod 4us,
		$$
		that is,
		\begin{equation}
    \label{poop6}
\ell(i+b+2j)\equiv -\ell+ r+(q+\ell)/2\mod s.
\end{equation}

		As for (\ref{poop2}), by \cite[Theorem 3.1 and Corollary 3.3]{MS3} and (\ref{corch}), we have
		$$
		[A^s BA^{2sa}B^{2sb},A^{1+s+2si}B^{2sj}]\equiv [A^{1+s+2si}B^{2sj},A^s BA^{2sa}B^{2sb}]^{-1}\equiv
		A^{-u} C^{-1 -s - 2s(i+b)}\mod Z(J_2(\alpha)).
		$$
Thus
$$
(A^s BA^{2sa}B^{2sb})^{[A^s BA^{2sa}B^{2sb},A^{1+s+2si}B^{2sj}]}=(A^s BA^{2sa}B^{2sa})^{A^{-u} C^{-1 -s - 2s(i+b)}}.
$$
Appealing to \cite[Theorem 3.1]{MS3}, we obtain
\begin{equation}
    \label{poop2l}
(A^s BA^{2sa}B^{2sb})^{[A^s BA^{2sa}B^{2sb},A^{1+s+2si}B^{2sj}]}=A^{s+2sa} B^{1+2s\ell+2sb}C^{-2sa+u}
A^{-4u\ell-6u\ell a-4u\ell(i+2b)+us\ell}.
\end{equation}

Due to (\ref{poopicat}), $A^s B A^{2sa}B^{2sb}=A^{s+2sa}B^{1+2sb} C^{-2sa}A^{-2u\ell a}$, where $A^{-2u\ell a}$ is central and we have
$(A^{-2u\ell a})^{1+2s\ell'}=A^{-2u\ell a}$. Applying \cite[Theorem 3.5]{MS3} to
$$(A^{s+2sa}B^{1+2sb} C^{-2sa})^{1+2s\ell'}= A^{expA + \xi} B^{expB} C^{expC}$$
gives
$$expA\equiv s + 2sa+2u\ell'+4u\ell' a+u^2-us\ell'\ell-2usa\mod 4us,$$
 $expB\equiv 1+2s\ell'+2sb+4u\ell'b+2us\mod 4us$,
    $expC\equiv -2sa-u\ell'-2u\ell' a\mod 4u$,
    $\xi\equiv u^2\mod 4us$, so
	\begin{equation}
    \label{poop2r}
(A^s B A^{2sa}B^{2sb})^{1+2s\ell'}=
A^{s + 2sa}	B^{1+2s\ell'+2sb}C^{-2sa-u\ell'}A^{2u(\ell'-s)+4u\ell'(a-b)-us\ell'\ell-2u\ell a}.
\end{equation}

Thus (\ref{poop2}) holds if and only if (\ref{poop2l}) equals (\ref{poop2r}).
Here $B^{2s\ell'}=B^{2s\ell}B^{2uq}=B^{2s\ell}A^{-2uq}$, with $A^{2u}$ central.  As $B^{2sb}\in Z_3(J_2(\alpha))$,
which is abelian, we see that (\ref{poop2}) holds if and only if
$$
C^u A^{-4u\ell-6u\ell a-4u\ell(i+2b)+us\ell}=C^{-u\ell'}A^{-2uq+2u(\ell'-s)+4u\ell'(a-b)-us\ell'\ell-2u\ell a}.
$$
Here $C^{u(1+\ell')}=A^{us(1+\ell')}$. As $A$ has order $4us$, the above condition becomes
$$
s(1+\ell')-4\ell-6\ell a-4\ell(i+2b)+s\ell\equiv -2q+2(\ell'-s)+4\ell'(a-b)-s\ell'\ell-2\ell a\mod 4s,
$$
which simplifies to
\begin{equation}
    \label{poop8}
\ell(i+b+2a)\equiv (q-3\ell)/2\mod s,
\end{equation}
where $q-3\ell$ is even. Subtracting (\ref{poop8}) from (\ref{poop6}), we obtain
\begin{equation}
    \label{poop9}
2\ell(j-a)\equiv r+\ell\mod s.
\end{equation}
This has a solution if and only if $m=2$.  If $m=2$, once $j$ and $a$ are chosen so that (\ref{poop9}) holds,
the equations (\ref{poop6}) and (\ref{poop8}) become identical and they can be solved for $i$ and $b$.
\end{proof}

\section{Necessity for $H_2(\alpha)$}\label{necfh2}

We assume in this section that $p=2$ and $m>2$. We set $s=2^{m-1}$, $u=s^2$, and $r=s/2$.

Note that $\alpha$ and $\alpha'$ have the same order modulo $2^{2m-1}$, namely $2^{m-1}$. It is
well-known that $[\Z/2^{2m-1}\Z]^\times$ is the internal direct product of the cyclic subgroups generated by
the classes of 5 and $-1$. Since $\alpha$ and $\alpha'$ are congruent to 1 modulo 4, their classes must belong
to the subgroup generated by the class of 5. In a finite cyclic group there is at most one subgroup of any given order,
so $\alpha$ and $\alpha'$ generate the same subgroup modulo $2^{2m-1}$, which means that there is an odd integer $t$
such that $\alpha'\equiv\alpha^t\mod 2^{2m-1}$. We may take $t$ to be positive and we do so.

By \cite[Section 6]{MS3},
$
H_2(\alpha)=\langle A,B\,|\, A^{[A,B]}=A^{\alpha},\, B^{[B,A]}=B^{\alpha},  A^{2^{2m-1}}=1=B^{2^{2m-1}}\rangle
$
and
$
H_2(\alpha')=\langle X,Y\,|\, X^{[X,Y]}=X^{\alpha'},\, Y^{[Y,X]}=Y^{\alpha'},  X^{2^{2m-1}}=1=Y^{2^{2m-1}}\rangle.
$
Moreover, by \cite[Section 5]{MS3}, we have
$
K_2(\alpha)=\langle a,b\,|\, a^{[a,b]}=a^\alpha,\, b^{[b,a]}=b^\alpha, a^{2^{2m-1}}=1=b^{2^{2m-1}}, [a,b]^{2^{m-1}}=1\rangle
$
and
$
K_2(\alpha')=\langle x,y\,|\, x^{[x,y]}=x^{\alpha'},\, y^{[y,x]}=y^{\alpha'},  x^{2^{2m-1}}=1=y^{2^{2m-1}},[x,y]^{2^{m-1}}=1\rangle.
$
There are automorphisms, say $\nu$ and $\mu$, of $H_2(\alpha')$ and $K_2(\alpha')$, given by $X\leftrightarrow Y$
and $x\leftrightarrow y$.

Consider the projection homomorphisms $\pi_\alpha:H_2(\alpha)\to K_2(\alpha)$, given by $A\mapsto a$ and $B\mapsto b$,
and $\pi_{\alpha'}:H_2(\alpha')\to K_2(\alpha')$, given by $X\mapsto x$ and $Y\mapsto y$.

We have the natural homomorphism $\Lambda:\mathrm{Aut}(H_2(\alpha'))\to \mathrm{Aut}(K_2(\alpha'))$ such that
$\pi_{\alpha'} \beta^\Lambda= \beta\pi_{\alpha'}$
for all $\beta\in \mathrm{Aut}(H_2(\alpha'))$.

\begin{prop}\label{autexcel}
    For any $u,v\in Z_2(K_2(\alpha'))$, the assignment $x\mapsto xu,\; y\mapsto yv$ extends to an automorphism
		 of $K_2(\alpha')$. In particular, the assignments
		$$
		x\mapsto x^{1+s},y\mapsto y;\;
    x\mapsto x, y\mapsto y^{1+s};\; x\mapsto xy^{s},y\mapsto y;\; x\mapsto x,y\mapsto yx^s;\;
    $$
		extend to automorphisms $\psi_1,\psi_2,\psi_3$, and $\psi_4$ of $K_2(\alpha')$.
\end{prop}

\begin{proof} This is shown in \cite[Proposition 5.9]{MS3}.
\end{proof}

\begin{prop}\label{newauto}
    The assignment $x\mapsto xy^r,\; y\mapsto y$ extends to an automorphism $\Phi$ of~$K_2(\alpha')$.
\end{prop}

\begin{proof} This is shown in \cite[Proposition 5.12]{MS3}.
\end{proof}

Given an odd integer $n\in\Z$ with inverse $q\in\Z$ modulo $2^{2m-1}$, \cite[Proposition 5.10]{MS3} ensures
that the assignment
\begin{equation*}
x\mapsto x^n,\; y\mapsto y^q
\end{equation*}
extends to an automorphism, say $f_n$, of $K_2(\alpha')$.


\begin{theorem}\label{indexh} Let $U=\{\psi_1^d\psi_2^e f_n\Phi^g\,|\, 1\leq n\leq m-1, n\equiv 1\mod 2, 0\leq d,e,g\leq 1\}$.
Then $K_2(\alpha')=H_2(\alpha')^\Lambda\, U$.
\end{theorem}

\begin{proof} Set $K=K_2(\alpha')$ and $H=H_2(\alpha')$. By \cite[Theorem 5.13]{MS3} and its proof, we have
$$
\mathrm{Aut}(K)/\mathrm{Aut}_2(K)=\langle \overline{\mu}\rangle\ltimes T,
$$
where $T\cong (\Z/2^{m-1}\Z)^\times \times (\Z/2\Z)^2$ is given by
$$T=\{\overline{f_n}\,|\, 1\leq n\leq m-1, n\equiv 1\mod 2\}\times\{1,\overline{\Phi}\}\times\{1,\overline{\Phi^\mu}\}.
$$
Here $\overline{f_{r+1}}$ has order 2, so
$$
T=\{1,\overline{f_{r+1}\Phi\Phi^\mu}\}\times\{\overline{f_n}\,|\, 1\leq n\leq m-1, n\equiv 1\mod 2\}\times\{1,\overline{\Phi}\}.
$$
Thus,
$$
\mathrm{Aut}(K)=\{1, \mu\}\{1,f_{r+1}\Phi\Phi^\mu\}\mathrm{Aut}_2(K)\{f_n\,|\, 1\leq n\leq m-1, n\equiv 1\mod 2\}
\{1,\Phi\}.
$$

Let $\Sigma\in \mathrm{Aut}(H)$ be defined by $X\mapsto X^{1+r}Y^r$, $Y\mapsto Y^{1+r}X^r$, as in \cite[Corollary 6.12]{MS3}.
Then $\overline{\Sigma^\Lambda}=\overline{f_{r+1}\Phi\Phi^\mu}$, that is, $\Sigma^\Lambda\equiv f_{r+1}\Phi\Phi^\mu\mod \mathrm{Aut}_2(K)$,
so
\begin{equation}
\label{deck}
\mathrm{Aut}(K)=(\{1, \nu\}\{1,\Sigma\})^\Lambda \mathrm{Aut}_2(K)\{f_n\,|\, 1\leq n\leq m-1, n\equiv 1\mod 2\}
\{1,\Phi\}.
\end{equation}

On the other hand, by \cite[Proposition 5.16]{MS3}, we have
$$
\overline{\langle \psi_1, \psi_2, \psi_3, \psi_4\rangle}=\mathrm{Aut}_2(K)/\mathrm{Inn}(K)\mathrm{Aut}_1(K)\cong (\Z/2\Z)^4,
$$
so
$$
\mathrm{Aut}_2(K)=\mathrm{Inn}(K)\mathrm{Aut}_1(K)\{1,\psi_1\}\{1,\psi_2\}\{1,\psi_3\}\{1,\psi_4\}.
$$
Here
$$
(\mathrm{Inn}(H)\mathrm{Aut}_2(H))^\Lambda=\mathrm{Inn}(K)\mathrm{Aut}_1(K)
$$
by \cite[Proposition 7.2]{MS3}. Moreover, if $\Gamma\in\mathrm{Aut}(H)$ is defined by $X\mapsto X^{1+s}$, $Y\mapsto YX^s$, as in
\cite[Proposition 6.8]{MS3}, we readily see that $\Gamma^\Lambda=\psi_1\psi_4$ and
$(\Gamma^\nu)^\Lambda=\psi_2\psi_3$. Therefore,
\begin{equation}
\label{deck2}
\mathrm{Aut}_2(K)=(\mathrm{Inn}(H)\mathrm{Aut}_2(H)\langle\Gamma,\Gamma^\nu\rangle)^\Lambda \{1,\psi_1\}\{1,\psi_2\}.
\end{equation}

Combining (\ref{deck}) and (\ref{deck2}) we obtain $K_2(\alpha')=H_2(\alpha')^\Lambda\, U$.
\end{proof}

\begin{prop}\label{isoK} We have $K_2(\alpha')\cong K_2(\alpha)$.
\end{prop}

\begin{proof} Consider the assignment
\begin{equation}
\label{tas}
x\mapsto a, y\mapsto b^t.
\end{equation}
We claim that (\ref{tas}) extends to an isomorphism $K_2(\alpha')\to K_2(\alpha)$. We have $[a,b]\in Z_2(K_2(\alpha))$ by
Proposition \ref{cenj}, so
$
[a,b^t]\equiv [a,b]^t\mod Z(K_2(\alpha)),
$
that is,
$
[a,b^t]=[a,b]^t z, z\in Z(K_2(\alpha)).
$
Hence
$$
a^{[a,b^t]}=a^{[a,b]^t z}=a^{[a,b]^t}=a^{\alpha^t}=a^{\alpha'},\; b^{[b^t,a]}=b^{[b,a]^t z^{-1}}=b^{[b,a]^t}=b^{\alpha^t}=b^{\alpha'}.
$$
This shows that the first two defining relations of $K_2(\alpha')$ are preserved. The next two defining relations of
$K_2(\alpha')$ are obviously preserved.  As for the fifth, we have $Z(K_2(\alpha))=\langle a^{2s},b^{2s}\rangle$ by Proposition \ref{cenj},
so
$$
[a,b^t]^{2^{m-1}}=([a,b]^t z)^{2^{m-1}}=[a,b]^{2^{m-1} t} z^{2^{m-1}}=1.
$$
Thus (\ref{tas}) extends to an homomorphism $K_2({\alpha'})\to K_2(\alpha)$. Likewise, if $q\in\Z$ is the inverse
of $t$ modulo $2^{2m-1}$, we have a homomorphism $K_2(\alpha)\to K_2({\alpha'})$ such that
$a\mapsto x, b\mapsto y^q$. As these are inverse of each other, we deduce $K_2(\alpha)\cong K_2({\alpha'})$.
\end{proof}

For $i\in\Z$, we set $\phi(i)=i(i-1)/2$.

\begin{theorem}\label{h2} We have $H_2(\alpha')\cong H_2(\alpha)$ if and only if $\alpha'\equiv\alpha\mod 2^{2m-2}$.
\end{theorem}

\begin{proof} Let $\gamma:K_2(\alpha')\to K_2(\alpha)$ be the isomorphism (\ref{tas}). For $u=\psi_1^d\psi_2^e f_n \Phi^g$,
if $g=0$, then
$$
xu\gamma=a^i, yu\gamma=b^j,
$$
for some odd integers $i$ and $j$, while if $g=1$, then
$$
xu\gamma=(ab^r)^i, yu\gamma=b^j,
$$
for some odd integers $i$ and $j$. We next apply Proposition~\ref{tool2} with $G_1=H_2(\alpha')$, $G_2=H_2(\alpha)$, $E=\{X,Y\}$, and $U$ as in Theorem \ref{indexh}.
This is a valid application thanks to Proposition \ref{autexcel}. For $g=0$, we take $h_{U,X}=A^i$ and $h_{U,Y}=B^j$,
and for $g=1$, we take $h_{U,X}=(AB^r)^i$ and $h_{U,Y}=B^j$. By Proposition~\ref{tool2}, $H_2(\alpha')\cong H_2(\alpha)$ if and only if
at least one of the assignments
\begin{equation}
\label{XA}
X\mapsto A^i, Y\mapsto B^j,
\end{equation}
\begin{equation}
\label{XAr}
X\mapsto (AB^r)^i, Y\mapsto B^j,
\end{equation}
extends to an isomorphism $H_2(\alpha')\to H_2(\alpha)$, where $i$ and $j$ are odd.

We proceed to prove that (\ref{XA}) extends to an isomorphism $H_2(\alpha')\to H_2(\alpha)$ if and only if
$i\equiv 1\mod 2s$, $j\equiv 1\mod 2s$, and $\ell'\equiv \ell\mod s$, while (\ref{XAr})
extends to an isomorphism $H_2(\alpha')\to H_2(\alpha)$ if and only if
$\ell'\equiv \ell\mod r$, $\ell'\not\equiv \ell\mod s$, $i\equiv 1\mod 2s$, and $j=1+kr$ with $k\equiv 1\mod 4$.
In particular, $H_2(1+8\ell)\cong H_2(9)$, where all integers $1+8\ell$ are congruent to 9 modulo $16$.

We set $C=[A,B]$ and observe from the proof of \cite[Theorem 8.1]{MS} that every element of $H_2(\alpha)$ can be written in one
and only one way as product of elements taken from $\langle A\rangle$, $\langle B\rangle$, and $\langle C\rangle$,
in any fixed order.

Note that (\ref{XA}) extends to an isomorphism if and only if
\begin{equation}
\label{XA2}
A^{[A^i,B^j]}=A^{\alpha'}, B^{[B^j,A^i]}=B^{\alpha'}.
\end{equation}

By \cite[Corollary 6.1]{MS3},
\begin{equation}
\label{XA3}
[A^i,B^j]=A^{-2s\ell j\phi(i)}B^{2s\ell i\phi(j)}C^{ij-2s\ell \phi(i)\phi(j)}.
\end{equation}
Using (\ref{XA3}) and \cite[Corollary 6.1]{MS3}, we see that
\begin{equation}
\label{XA4}
A^{[A^i,B^j]}=A^{1+2s\ell ij}C^{2s\ell i\phi(j)}.
\end{equation}
On the other hand,
\begin{equation}
\label{XA5}
A^{\alpha'}=A^{1+2s\ell'}.
\end{equation}
Thus (\ref{XA4}) is equal to (\ref{XA5}) if and only if $\ell ij\equiv \ell'\mod s$ and $j\equiv 1\mod 2s$.

Applying the automorphism $A\leftrightarrow B$ to (\ref{XA4}) and then making the substitution $i\leftrightarrow j$ yields
\begin{equation}
\label{XA6}
B^{[B^j,A^i]}=B^{1+2s\ell ij}C^{-2s\ell j\phi(i)}.
\end{equation}
On the other hand,
\begin{equation}
\label{XA7}
B^{\alpha'}=B^{1+2s\ell'}.
\end{equation}
Thus (\ref{XA6}) is equal to (\ref{XA7}) if and only if $\ell ij\equiv \ell'\mod s$ and $i\equiv 1\mod 2s$.

We deduce that (\ref{XA2}) holds if and only if $i\equiv 1\mod 2s$, $j\equiv 1\mod 2s$, and $\ell \equiv \ell'\mod s$.

Moving now to (\ref{XAr}), note that this assignment extends to an isomorphism if and only if
\begin{equation}
\label{P}
((AB^r)^i)^{[(AB^r)^i,B^j]}=(AB^r)^{i\alpha'}
\end{equation}
and
\begin{equation}
\label{Q}
B^{[B^j,(AB^r)^i]}=B^{\alpha'}.
\end{equation}
Indeed, the conditions are clearly necessary. Using \cite[Corollary 6.4]{MS3}, we see that
\begin{equation}
\label{P1}
(AB^r)^i=A^{i+u\phi(i)}B^{ir+\ell(r+1)\phi(i)u/2}C^{-r\phi(i)+\ell(r-1)\phi(i)u/2}.
\end{equation}
By \cite[Note 6.5]{MS3}, it follows that $((AB^r)^i)^u=1$. Thus if (\ref{P}) and (\ref{Q}) hold,
all 4 defining relations of $H_2(\alpha)$ are preserved, so  (\ref{XAr}) extends to an homomorphism.
Using (\ref{P1}) and \cite[Corollary 6.3]{MS3}, we find that
\begin{equation}
\label{P2}
[(AB^r)^i,B^j]=A^{-2s\ell\phi(i)j}B^{2s\ell i\phi(j)+u(1-\phi(i))}C^{ij+u\phi(i)j},
\end{equation}
where $C^s$ is central in $H_2(\alpha)$ by Proposition \ref{cenj}. The given homomorphism will then map $[X,Y]^s$ to $C^{ijs}$.
As $i$ and $j$ are odd, the given homomorphism will restrict to an isomorphism between the corresponding
centers, and it will then be an isomorphism.

By means of (\ref{P1}) and (\ref{P2}), a careful use of \cite[Corollary 6.1]{MS3} yields
\begin{equation}
\label{P3}
((AB^r)^i)^{[(AB^r)^i,B^j]}=A^{i+u\phi(i)+2s\ell i^2 j}B^{u+ir+\ell(r+1)\phi(i)u/2}C^{-r\phi(i)+\ell(r-1)\phi(i)u/2+u\phi(i)+
2s\ell i^2\phi(j)+u(1-\phi(i))}.
\end{equation}

On the other hand, \cite[Corollary 6.4]{MS3} gives
\begin{equation}
\label{P4}
(AB^r)^{i\alpha'}=A^{i+u\phi(i)+2s\ell' i}B^{u+ir+u+\ell(r+1)\phi(i)u/2}C^{-r\phi(i)+\ell(r-1)\phi(i)u/2-u\ell'\phi(i)-
i^2\ell' u/2}.
\end{equation}

By \cite[Corollary 6.1]{MS3} all factors appearing in (\ref{P2}) commute with each other, so
\begin{equation}
\label{P5}
[B^j,(AB^r)^i]=A^{2s\ell\phi(i)j}B^{-2s\ell i\phi(j)+u(1-\phi(i))}C^{-ij-u\phi(i)j}.
\end{equation}
It follows from (\ref{P5}) and \cite[Corollary 6.1]{MS3} that
\begin{equation}
\label{P6}
B^{[B^j,(AB^r)^i]}=B^{1-2s\ell ij}C^{-2s\ell\phi(i)j}.
\end{equation}
On the other hand, we have
\begin{equation}
\label{P7}
B^{\alpha'}=B^{1+2s\ell'}.
\end{equation}

We readily see that (\ref{P6}) is equal to  (\ref{P7}) if and only if $i\equiv 1\mod 2s$ and $\ell j\equiv \ell'\mod s$.
On the other hand, (\ref{P3}) is equal to (\ref{P4}) if and only if $\ell ij\equiv \ell'\mod s$ and
$$
2s\ell i^2 \phi(j) + u\equiv u \phi(i)-i^2\ell' u/2\mod 2u.
$$
Thus (\ref{P6}) is equal to  (\ref{P7}) and (\ref{P3}) is equal to (\ref{P4}) if and only if
\begin{equation}
\label{P8}
i\equiv 1\mod 2s,\; \ell j\equiv \ell'\mod s,\; 2s\ell \phi(j) + u\equiv -\ell' u/2\mod 2u.
\end{equation}

Suppose first that (\ref{P8}) holds. Then
$$
2s\ell \phi(j)\equiv 0\mod u/2,
$$
which is equivalent to
$$
j\equiv 1\mod r.
$$
Thus $\ell\equiv\ell'\mod r$, so that $\ell'=\ell+qr$ for some $q\in\Z$.
Writing $j=1+rk$, with $k\in\Z$, and replacing this in (\ref{P8}) yields
\begin{equation}
\label{P9}
\ell j k u/2 + u \equiv -\ell' u/2\mod 2u.
\end{equation}
Note that (\ref{P9}) implies
$
\ell j k\equiv \ell'\mod 2,
$
so $k$ must be odd. Also, from $\ell j\equiv \ell'\mod s$, we obtain
$$
\ell (1+rk)\equiv \ell+qr\mod s.
$$
As $\ell$ is odd, this implies
$r k\equiv qr\mod s$,
or
$
k\equiv q\mod 2,
$
so $q$ must be odd too. Going back to (\ref{P9}), we may write it in the form
$
\ell j k + 2 \equiv -\ell'\mod 4,
$
or
$
\ell j k + 2 \equiv -\ell-rq\mod 4,
$
that is
$
\ell(1+jk)\equiv 2-rq\mod 4.
$
Here $j$, $k$, and and $\ell$ are odd, so the last condition becomes
$
1+jk\equiv 2-rq\mod 4.
$
If $m>3$ this translates into
$$
1+(1+rk)k\equiv 1+k\equiv 2\mod 4,
$$
so $k\equiv 1\mod 4$. If $m=3$, the fact that $q$ is odd implies
$
1+(1+2k)k\equiv 0\mod 4.
$
As $k^2\equiv 1\mod 4$, we deduce
$
1+k+2\equiv 0\mod 4,
$
so
$
k\equiv 1\mod 4.
$
Retracing our steps we see that the conditions $\ell'\equiv \ell\mod r$, $\ell'\not\equiv \ell\mod s$, $i\equiv 1\mod 2s$, and $j=1+kr$ with
$k\equiv 1\mod 4$, are sufficient.
\end{proof}

\section{Necessity for $J_2(\alpha)$ Part II}\label{siete}

We assume in this section that $p=2$ and $m>2$. We set $s=2^{m-1}$, $u=s^2$, $r=s/2$, and $\overline{r}=r/2$,
and maintain the notation introduced in Section \ref{part1}.


\begin{theorem}\label{necj22} Suppose that $\alpha'\equiv\alpha\mod 2^{2m-2}$ but
$\alpha'\not\equiv\alpha\mod 2^{2m-1}$. Then $J_2(\alpha')\not\cong J_2(\alpha)$.
\end{theorem}

\begin{proof} Suppose, if possible, that $J_2(\alpha')\cong J_2(\alpha)$. The proof of Theorem \ref{h2} provides an explicit isomorphism
$\gamma:H_2(\alpha')\to H_2(\alpha)$, given by
$$
X_0\mapsto A_0B_0^r, Y_0\mapsto B_0^{1+r}.
$$
We next apply Proposition~\ref{tool2} with $G_1=J_2(\alpha')$, $G_2=J_2(\alpha)$, $E=\{X,Y\}$, and $U$ as in Theorem~\ref{indexj}.
This is a valid application thanks to Proposition \ref{autjo}. Arguing as in the proof of Theorem~\ref{necj2}, we see that
for some choice of integers $i,j,a,b$, the assignment
$$
X\mapsto (AB^r)^{1+s}A^{2si}B^{2sj}, Y\mapsto B^{1+r} (AB^r)^s A^{2sa}B^{2sb}
$$
extends to an isomorphism $J_2(\alpha')\to J_2(\alpha)$. This implies that the following equations hold:
\begin{equation}\label{eq.LR.1}
    ((AB^r)^{1+s}A^{2si}B^{2sj})^{[(AB^r)^{1+s}A^{2si}B^{2sj},B^{1+r} (AB^r)^s A^{2sa}B^{2sb}]}
    = ((AB^r)^{1+s}A^{2si}B^{2sj})^{\alpha'},
\end{equation}
\begin{equation}\label{eq.LR.2}
    (B^{1+r} (AB^r)^s A^{2sa}B^{2sb})^{[B^{1+r} (AB^r)^s A^{2sa}B^{2sb},(AB^r)^{1+s}A^{2si}B^{2sj}]}
    = (B^{1+r} (AB^r)^s A^{2sa}B^{2sb})^{\alpha'}.
\end{equation}


The calculation of both sides of \eqref{eq.LR.1} and \eqref{eq.LR.2} are tedious and they appear in the Appendix.

From now on, we set $\varphi(n)=n(n-1)(n-2)/6$, for $n\in\Z$. From \eqref{eq.L.1} and \eqref{eq.R.1} in the Appendix, we get that \eqref{eq.LR.1} is equivalent to
\begin{align*}
    &A^{2s(\ell - \ell') + u\ell} B^{r + sr + 2sj + u\ell(\overline{r} - 1)} C^{sr(\ell + \ell') + u(\overline{r} + 1)} B^{-r - sr - 2sj - u(\overline{r}\ell + \ell')}\\
    &= A^{2u(2i\ell' - 2j\ell' + \ell' - 4i\ell - 2j\ell - 2b\ell - 2\ell) - ur\ell(\ell' + \ell + 1) + us\overline{r}(\ell' + 1) + u\varphi(2s\ell') + 2us(i + j + a + b + 1)}.
\end{align*}

By \cite[Theorem 3.1]{MS3} we have
$[C^{sr(\ell + \ell') + u(\overline{r} + 1)},B^{-r - sr - 2sj - u(\overline{r}\ell + \ell')}] = 1$, so
\begin{align*}
    &A^{2s(\ell - \ell') + u\ell} B^{-u(\ell + \ell')} C^{sr(\ell + \ell') + u(\overline{r} + 1)}\\
    &= A^{2u(2i\ell' - 2j\ell' + \ell' - 4i\ell - 2j\ell - 2b\ell - 2\ell) - ur\ell(\ell' + \ell + 1) + us\overline{r}(\ell' + 1) + u\varphi(2s\ell') + 2us(i + j + a + b + 1)},
\end{align*}
and from this
\begin{equation}\label{eq.task.1}
    \begin{split}
        &A^{2s(\ell - \ell') + u\ell} C^{sr(\ell + \ell') + u(\overline{r} + 1)}\\
        &= A^{-u(\ell + \ell') + 2u(2i\ell' - 2j\ell' + \ell' - 4i\ell - 2j\ell - 2b\ell - 2\ell) - ur\ell(\ell' + \ell + 1) + us\overline{r}(\ell' + 1) + u\varphi(2s\ell') + 2us(i + j + a + b + 1)}.
    \end{split}
\end{equation}

From \eqref{eq.L.2} and \eqref{eq.R.2} in the Appendix, we get that \eqref{eq.LR.2} is equivalent to
\begin{equation}\label{eq.task.2}
    B^{2s(\ell - \ell') - u\ell'}
    = A^{2u(2a\ell' + 2a\ell - 2b\ell' + 4b\ell + 2i\ell + 3\ell + \ell') + ur\ell - us\ell(\ell' - 1) + 2us}.
\end{equation}

Since $\alpha\equiv\alpha'\mod u$ and $\alpha\not\equiv\alpha'\mod 2u$, this is, $\ell\equiv\ell'\mod r$ and $\ell\not\equiv\ell'\mod s$, then $\ell' = \ell - r - s\ell'_1$ for some $\ell'_1\in\Z$ and
\[
B^{2s(\ell - \ell') - u\ell'}
= B^{u(1 - \ell) + 2u\ell'_1 + ur + us\ell'_1}
= A^{u(\ell - 1) - 2u\ell'_1 - ur - us\ell'_1}.
\]

Replacing the above in \eqref{eq.task.2} produces
\[
A^{u(1 - \ell) + 4u\ell(i + 2a + b + 2) + 2u\ell'_1 + ur(\ell + 1) + us(\ell + \ell'_1) + usr + 2us(a + b + \ell'_1)}
= 1,
\]
and from this
\begin{equation}\label{cong.LR.1}
    u(1 - \ell) + 4u\ell(i + 2a + b + 2) + 2u\ell'_1 + ur(\ell + 1) + us(\ell + \ell'_1) + usr + 2us(a + b + \ell'_1)\equiv 0\mod 4us.
\end{equation}

In the other hand,
\[
A^{2s(\ell - \ell') + u\ell} C^{sr(\ell + \ell') + u(\overline{r} + 1)}
= A^{u(\ell + 1) + 2u\ell'_1} C^{u(\ell + 1) + ur\ell'_1}
= A^{u(\ell + 1) + 2u\ell'_1 + us(\ell + 1) + usr\ell'_1}.
\]


Replacing the above in \eqref{eq.task.1} produces
\[
A^{-u(\ell + 1)  - 2u\ell'_1 - 4u\ell(i + 2j + b + 1) - ur(\ell - 1) + us(\overline{r} - \ell + \ell'_1 + 1) + usr(\overline{r} + 1) + u\varphi(2s\ell') + 2us(a + b + \ell'_1 + 1)}
= 1,
\]
and from this
\begin{equation}\label{cong.LR.2}
    \begin{split}
        &-u(\ell + 1)  - 2u\ell'_1 - 4u\ell(i + 2j + b + 1) - ur(\ell - 1) + us(\overline{r} - \ell + \ell'_1 + 1)\\
        &\qquad + usr(\overline{r} + 1) + u\varphi(2s\ell') + 2us(a + b + \ell'_1 + 1)\equiv 0\mod 4us.
    \end{split}
\end{equation}

Subtracting \eqref{cong.LR.2} from \eqref{cong.LR.1} produces
\[
2u + 4u\ell(2i + 2j + 2a + 2b + 3) + 4u\ell'_1 + us(\ell - \overline{r} - 1) + usr\overline{r} + u\varphi(2s\ell')
\equiv0\mod 4us.
\]

Dividing this by $2u$ and reducing modulo 2 yields
\[
1 + 2\ell(2i + 2j + 2a + 2b + 3) + 2\ell'_1 + r(\ell - \overline{r} - 1) + s\overline{r}^2 + \varphi(2s\ell')/2
\equiv0\mod 2,
\]
which is impossible.

\end{proof}

\section{Common structural features of $J_1(\alpha)$ and $J_1(\alpha')$}

We assume throughout this section that we are in Case 1.

\begin{theorem}\label{csf1} All proper terms and all factors of the upper central series of $J_1(\alpha)$ and $J_1(\alpha')$ are isomorphic.
\end{theorem}

\begin{proof} Set $J=J_1(\alpha)$, so that
$J=\langle A,B\,|\, A^{[A,B]}=A^{\alpha},\, B^{[B,A]}=B^{\alpha}, A^{p^{3m}}=B^{p^{3m}}=1\rangle$. Then
$$
Z_1(J)=\langle A^{p^{2m}}\rangle, Z_2(J)=\langle A^{p^{2m}},C^{p^{m}}\rangle,
Z_3(J)=\langle  A^{p^{m}}, B^{p^{m}}, C^{p^{m}}\rangle,
$$
$$
Z_4(J)=\langle  A^{p^{m}}, B^{p^{m}}, C\rangle, Z_5(J)=J,
$$
by \cite[Theorem 8.1]{MS}. Here $Z_1(J)\cong\Z/p^m\Z$ by
\cite[Proposition 7.2]{MS}. Since $Z_3(J)$ is abelian by \cite[Proposition 9.1]{MS}, we infer
$Z_2(J)\cong\Z/p^m\Z\times \Z/p^m\Z$ and $Z_2(J)/Z_1(J)\cong\Z/p^m\Z$ by
\cite[Proposition 7.2]{MS}. In the proof of \cite[Theorem 11.1]{MS} one constructs a group
$\langle X_0,Y_0\rangle$ isomorphic to $J$ via $A\mapsto X_0$, $B\mapsto Y_0$. Under this isomorphism
$Z_3(J)$ corresponds to a subgroup of order $p^{4m}$ generated by elements $X,Y,Z$ subject to the
defining relations:
$$
XY=YX,\; YZ=ZY,\; XZ=ZX,\; X^{p^{2m}}=1, X^{p^{m}}Y^{p^{m}}=1,\;  Z^{p^{m}}=1.
$$
In particular, the isomorphism type of $Z_3(J)$ is independent of $\ell$. In view of the fundamental
relation $A^{p^{2m}}B^{p^{2m}}=1$, it follows that $Z_3(J)/Z_2(J)$ is a quotient of $\Z/p^m\Z\times \Z/p^m\Z$.
But $Z_3(J)/Z_2(J)$ has order $p^{2m}$, so $Z_3(J)/Z_2(J)\cong\Z/p^m\Z\times \Z/p^m\Z$. As revealed by the proof of \cite[Theorem 11.1]{MS},
the isomorphism
$J\to \langle X_0,Y_0\rangle$ sends $Z_4(J)$ onto a subgroup of order $p^{5m}$ generated by elements $X,Y,Z_0$ subject to the defining relations
$$
X^{Z_0}=X^\alpha,\; Y^{Z_0}=Y^{\beta},\; XY=YX,\; X^{p^{2m}}=1,\; X^{p^{m}}Y^{p^{m}}=1,
Z_0^{p^{2m}}=1,
$$
where $\beta=2-\alpha$ is the inverse of $\alpha$ modulo $p^{2m}$. Likewise, $Z_4(J(\alpha'))$ has order $p^{5m}$ and
is generated by elements $x,y,z_0$ subject to the defining relations
$$
x^{z_0}=x^{\alpha'},\; y^{z_0}=y^{\beta'},\; xy=yx,\; x^{p^{2m}}=1,\; x^{p^{m}}y^{p^{m}}=1,
z_0^{p^{2m}}=1,
$$
where $\beta'=2-\alpha'$ is the inverse of $\alpha'$ modulo $p^{2m}$. Now $\alpha$ and $\alpha'$ have the same order $p^m$ modulo $p^{2m}$,
so $\alpha\equiv (\alpha')^i\mod p^{2m}$ for some $i>0$ relatively prime to $p$. It follows that the assignment
$X\mapsto x$, $Y\mapsto y$, $Z_0\mapsto z_0^i$ extends to a group isomorphism $\langle X,Y,Z_0\rangle\to \langle x,y,z_0\rangle$,
so $Z_4(J(\alpha))\cong Z_4(J(\alpha'))$.  It is clear that $Z_4(J)/Z_3(J)$ is a quotient of $\Z/p^m\Z$.
But $Z_4(J)/Z_3(J)$ has order $p^m$, so $Z_4(J)/Z_3(J)\cong\Z/p^m\Z$. It is also clear that
$J/Z_4(J)$ and $\Z/p^m\Z\times \Z/p^m\Z$ are quotients of each other, so $J/Z_4(J)\cong \Z/p^m\Z\times \Z/p^m\Z$.
\end{proof}

\section{Common structural features of $J_2(\alpha)$ and $J_2(\alpha')$}

We assume throughout this section that we are in Case 2.

\begin{theorem}\label{csf2} The derived subgroups and all factors of the upper central series of $J_2(\alpha)$ and $J_2(\alpha')$ are isomorphic.
\end{theorem}

\begin{proof} If $m=1$ then $J_2(\alpha)\cong J_2(\alpha')$ by Theorem \ref{suf}, so we assume $m>1$.

Set $J=J_2(\alpha)$, so that
$J=\langle A,B\,|\, A^{[A,B]}=A^{\alpha},\, B^{[B,A]}=B^{\alpha}, A^{2^{2m-1}}=B^{2^{2m-1}}=1\rangle$. The terms
of the upper central series of $J$ are given in Proposition \ref{cenj}. Here $Z_1(J)\cong\Z/2^m\Z$ by
\cite[Proposition 7.2]{MS}. Since $Z_3(J)$ is abelian by \cite[Proposition 9.2]{MS}, we infer
$Z_2(J)\cong\Z/2^m\Z\times \Z/2^m\Z$ and $Z_2(J)/Z_1(J)\cong\Z/2^m\Z$ by \cite[Theorem 7.1]{MS},
\cite[Proposition 7.2]{MS}, and the fundamental relation $A^{2^{3m-2}}=C^{2m-1}$. In \cite[Theorem 11.2]{MS} one constructs a
group $\langle x_0,y_0\rangle$ isomorphic to $J$ via $A\mapsto x_0$, $B\mapsto y_0$. Under this isomorphism
$Z_3(J)$ corresponds to a subgroup of order $2^{4m-2}$ generated by elements $x,y,z$ subject to the
defining relations
\[
    xy=yx,\;
    xz=zx,\;
    yz=zx,\;
    z^{2^m}=x^{2^{2m-2}},\;
    x^{2^{m-1}}y^{2^{m-1}}=1,\;
    x^{2^{2m-1}}=1.
    \]
In particular, the isomorphism type of $Z_3(J)$ is independent of $\ell$. In view of the fundamental
relation $A^{2^{2m-1}}B^{2^{2m-1}}=1$, it follows that $Z_3(J)/Z_2(J)$ is a quotient of $\Z/2^{m-1}\Z\times \Z/2^{m-1}\Z$.
But $Z_3(J)/Z_2(J)$ has order $2^{2(m-1)}$, so $Z_3(J)/Z_2(J)\cong \Z/2^{m-1}\Z\times \Z/2^{m-1}\Z$.

It is clear that $J/Z_4(J)$ and $\Z/2^{m-1}\Z\times \Z/2^{m-1}\Z$ are quotients of each other, so
$J/Z_4(J)\cong\Z/2^{m-1}\Z\times \Z/2^{m-1}\Z$. Since $|J|=2^{7m-3}$, it follows that $|Z_4(J)|=2^{5m-1}$.
It is also clear that $Z_4(J)/Z_3(J)$ is a quotient of $\Z/2\Z\times\Z/2\Z\times\Z/2^m\Z$.
But $Z_4(J)/Z_3(J)$ has order $2^{m+1}$, so $Z_4(J)/Z_3(J)\cong \Z/2\Z\times\Z/2\Z\times\Z/2^m\Z$.

Set $s=2^{m-1}$, $u=s^2$, and $r=s/2$. By \cite[Theorem 8.1]{MS}, we have $[J,J]=\langle A^{2s},B^{2s},C\rangle$.
Here $J/[J,J]$ and $\Z/2^{m}\Z\times \Z/2^{m}\Z$ are quotients of each other, so
$J/[J,J]\cong\Z/2^{m}\Z\times \Z/2^{m}\Z$, whence $|[J,J]|=2^{5m-3}$. Consider the group
\[
U=\langle X,Y,Z\,|\,
    [X,Y]=1,\;
    X^Z=X^\alpha,\;
    Y^Z=Y^{\beta},\;
    Z^{2u}=X^{u},\;
    X^{s}Y^{s}=1,\;
    X^{2u}=1\rangle,
 \]
where $\beta=2-\alpha$ is the inverse of $\alpha$ modulo $2u$.
It is not difficult to see that $|U|\leq 2^{5m-3}$. Since $Z_3(J)$ is abelian, we have $[A^{2s},B^{2s}]=1$.
It follows that the map $X\mapsto A^{2s}$, $Y\mapsto B^{2s}$, $Z\mapsto C$ extends to an isomorphism $U\to [J,J]$.
Likewise, $Z_4(J(\alpha'))$ is isomorphic to
$$
V=
\langle a,b,c\,|\,
    [a,b]=1,\;
    a^c=a^{\alpha'},\;
    b^c=b^{\beta'},\;
    c^{2u}=a^{u},\;
    a^{s}b^{s}=1,\;
    a^{2u}=1\rangle,
		$$
where $\beta'=2-\alpha'$ is the inverse of $\alpha'$ modulo $2u$.
Now $\alpha$ and $\alpha'$ have the same order $s$ modulo~$2u$, so as explained at the beginning of Section \ref{necfh2},
$\alpha\equiv (\alpha')^i\mod 2u$ for some odd $i>0$. It follows that the assignment
$X\mapsto a$, $Y\mapsto b$, $Y\mapsto c^i$ extends to a group isomorphism $U\to V$,
so $[J(\alpha),J(\alpha)]\cong [J(\alpha'),J(\alpha')]$.
\end{proof}


\section{Proof of Theorems A through E}

We begin with following observations: Theorem A follows from Theorems \ref{suf}, \ref{exel}, \ref{csf1}, and Proposition \ref{suf2}.
Theorem B is consequence of Theorems \ref{suf}, \ref{necj2}, \ref{h2}, \ref{necj22}, \ref{csf2}, and Proposition~\ref{isoK}.
Theorem C is contained in Theorem \ref{suf}. Regarding Theorem E, if $m=1$ or $m=2$ we appeal to Theorem B, while
if $m>2$, we resort to Theorem \ref{h2}.

We proceed to prove Theorem D. It is shown in \cite{M} that $G(\alpha)$ is infinite if and only if $\alpha=1$,
and that $G(\alpha)$ is trivial if and only if $\alpha=0$ or $\alpha=2$. Thus,
we may assume in what follows that $\alpha$ and $\beta$ are different and do not belong to $\{0,1,2\}$.

By \cite[Theorem 10.1]{MS}, the order of $G(\alpha)$ is wholly dependent on the prime factorization of $\alpha-1$.
Thus, if $G(\alpha)\cong G(\beta)$, then necessarily $\beta=2-\alpha$.

Suppose that $G(\alpha)\cong G(2-\alpha)$. In Case 1, $J_1(\alpha)\cong J_1(2-\alpha)$ and therefore $H_1(\alpha)\cong H_1(2-\alpha)$,
so $\alpha\equiv 2-\alpha\mod p^{2m}$, that is, $2(\alpha-1)\equiv 0\mod p^{2m}$, by Theorem A. But $v_{p}(\alpha-1)=m$,
so Case~1 does not occur. In Case 3, $J_3(\alpha)\cong J_3(2-\alpha)$, where
$\alpha=1+3\ell$, $2-\alpha=1-3\ell$ and $\gcd(3,\ell)=1$. If $\ell\equiv 1\mod 3$ then $-\ell\equiv -1\mod 3$, so
$|J_3(\alpha)|=3^7$ and $|J_3(2-\alpha)|=3^{10}$ by \cite[Theorem 10.1]{MS}, while if $\ell\equiv -1\mod 3$ then $-\ell\equiv 1\mod 3$, so
$|J_3(\alpha)|=3^{10}$ and $|J_3(2-\alpha)|=3^{7}$ by \cite[Theorem 10.1]{MS}. Thus, Case 3 does not occur. Therefore, Case 2 must occur,
$J_2(\alpha)\cong J_2(2-\alpha)$, with $\alpha=1+2^m$ or $\alpha=1-2^m$, so that $\alpha-(2-\alpha)=\pm 2^{m+1}$. It follows form
Theorem B that $m=1$ or $m=2$, that is, $\alpha\in\{-3,-1,3,5\}$.

Conversely, if $\alpha\in\{-3,-1,3,5\}$, then Theorem B ensures that $G(\alpha)\cong G(\beta)$.

Suppose next that $v_3(\alpha-1)\neq 1$. Then $v_3((2-\alpha)-1)\neq 1$, so \cite[Theorem 10.1]{MS} ensures that
$G(\alpha)$ and $G(2-\alpha)$ have the same order. Moreover, Theorems A and B, together with \cite[Theorem 8.1]{MS},
ensure that the derived subgroups and all factors of the upper central series of $G(\alpha)$ and $G(2-\alpha)$ are isomorphic.



\section{Appendix}

We proceed to calculate both sides of \eqref{eq.LR.1} and \eqref{eq.LR.2}. We maintain the notation of Section \ref{siete}.

Applying \cite[Theorem 3.5]{MS3} to
$(AB^r)^{1+s} = A^{expA} B^{expB} C^{expC} A^{\xi}$ produces
$expA\equiv 1 + s + u\varphi(s)\ell + usr - ur\ell\mod 4us$,
$expB\equiv r + sr + u\overline{r}\ell - ur\overline{r}\ell\mod 4us$,
$expC\equiv -s\overline{r} + sr\varphi(s) + u\overline{r}(\ell - 1)\mod 4u$,
$\xi\equiv ur\varphi(s) + usr\mod 4us$, so
\begin{equation}\label{eq.app.1}
    (AB^r)^{1+s}
    = A^{1 + s} B^{r + sr + u\overline{r}\ell} C^{-s\overline{r}} A^{u\varphi(s)\ell + ur\ell(\overline{r} - 1) + us\overline{r}(\ell - 1)}.
\end{equation}

Applying \cite[Theorem 3.2]{MS3} to
$(A^{1 + s} B^{r + sr + u\overline{r}\ell} C^{-s\overline{r}}) (A^{2si} B^{2sj} C^0) = A^{expA} B^{expB} C^{expC} A^{\xi}$ produces
$expA\equiv 1 + s + 2si - us\ell i\mod 4us$,
$expB\equiv r + sr + 2sj + u\overline{r}\ell + usri + us\ell i\mod 4us$,
$expC\equiv -s\overline{r} - ui + uri\mod 4u$, $\xi\equiv usri\mod 4us$, so
\begin{equation}\label{eq.app.2}
    (A^{1 + s} B^{r + sr + u\overline{r}\ell} C^{-s\overline{r}}) (A^{2si} B^{2sj} C^0)
    = A^{1 + s + 2si} B^{r + sr + 2sj + u\overline{r}\ell} C^{-s\overline{r} - ui} A^{2usi + usri}.
\end{equation}

By \eqref{eq.app.1} and \eqref{eq.app.2},
\begin{equation}\label{eq.imp.1}
    \begin{split}
        (AB^r)^{1+s} A^{2si} B^{2sj}
        &= (A^{1 + s} B^{r + sr + u\overline{r}\ell} C^{-s\overline{r}}) (A^{2si} B^{2sj} C^0) A^{u\varphi(s)\ell + ur\ell(\overline{r} - 1) + us\overline{r}(\ell - 1)}\\
        &= A^{1 + s + 2si} B^{r + sr + 2sj + u\overline{r}\ell} C^{-s\overline{r} - ui} A^{ur\ell(\overline{r} - 1) + u\varphi(s)\ell + us\overline{r}(\ell - 1) + usri + 2usi}.
    \end{split}
\end{equation}

Applying \cite[Theorem 3.5]{MS3} to
$(AB^r)^s = A^{expA} B^{expB} C^{expC} A^{\xi}$ produces
$expA\equiv s + u\varphi(s)\ell\mod 4us$,
$expB\equiv sr + ur\overline{r}\ell - u\overline{r}\ell\mod 4us$,
$expC\equiv -u\overline{r} + s\overline{r} + sr\varphi(s) + ur\overline{r}\mod 4u$,
$\xi\equiv ur\varphi(s) + us\overline{r}\ell^2\mod 4us$, so
\begin{equation}\label{eq.app.3}
    (AB^r)^s
    = A^s B^{sr - u\overline{r}\ell} C^{-u\overline{r} + s\overline{r}} A^{u\varphi(s)\ell - ur\overline{r}\ell + usr\overline{r} + us\overline{r}\ell^2}.
\end{equation}

Applying \cite[Theorem 3.2]{MS3} to
$(A^0 B^{1+r} C^0) (A^s B^{sr - u\overline{r}\ell} C^{-u\overline{r} + s\overline{r}}) = A^{expA} B^{expB} C^{expC} A^{\xi}$ produces
$expA\equiv s - u\ell - ur\ell + us\ell + usr\mod 4us$,
$expB\equiv 1 + r + sr - u\overline{r}\ell - ur\ell - us\overline{r}\ell\mod 4us$,
$expC\equiv s(\overline{r} - 1) - sr + u(\overline{r}\ell + \ell - \overline{r}) + ur(\overline{r} + 1)\mod 4u$,
$\xi\equiv us\ell^2(\overline{r} + 1) + usr\overline{r} + 2us\mod 4us$, so
\begin{equation}\label{eq.app.4}
    \begin{split}
        &(A^0 B^{1+r} C^0) (A^s B^{sr - u\overline{r}\ell} C^{-u\overline{r} + s\overline{r}})\\
        &= A^{s - u\ell} B^{1 + r + sr - u\overline{r}\ell} C^{s(\overline{r} - 1) - sr + u(\overline{r}\ell + \ell - \overline{r})} A^{us(\overline{r}\ell + \overline{r} + \ell + 1) + 2us}.
    \end{split}
\end{equation}

By \eqref{eq.app.3} and \eqref{eq.app.4},
\begin{equation}\label{eq.app.5}
    \begin{split}
        &B^{1+r} (AB^r)^s\\
        &\quad = (A^0 B^{1+r} C^0) (A^s B^{sr - u\overline{r}\ell} C^{-u\overline{r} + s\overline{r}}) A^{u\varphi(s)\ell - ur\overline{r}\ell + usr\overline{r} + us\overline{r}\ell^2}\\
        &\quad = A^{s - u\ell} B^{1 + r + sr - u\overline{r}\ell} C^{s(\overline{r} - 1) - sr + u(\overline{r}\ell + \ell - \overline{r})} A^{-ur\overline{r}\ell + us(\overline{r}\ell + \ell + 1) + usr(\overline{r} + 1) + u\varphi(s)\ell + 2us}.
    \end{split}
\end{equation}

Applying \cite[Theorem 3.2]{MS3} to
\[
(A^{s - u\ell} B^{1 + r + sr - u\overline{r}\ell} C^{s(\overline{r} - 1) - sr + u(\overline{r}\ell + \ell - \overline{r})}) (A^{2sa} B^{2sb} C^0)
= A^{expA} B^{expB} C^{expC} A^{\xi}
\]
produces
$expA\equiv s + 2sa - u\ell - 2u\ell a - us\ell a\mod 4us$,
$expB\equiv 1 + r + sr + 2sb - u\overline{r}\ell - us\ell a + usra\mod 4us$,
$expC\equiv s(\overline{r} - 1) - 2sa - sr + u(\overline{r}\ell + \ell - \overline{r} - a) + 2ua + ura\mod 4u$,
$\xi\equiv usra + 2usa\mod 4us$, so
\begin{equation}\label{eq.app.6}
    \begin{split}
        &(A^{s - u\ell} B^{1 + r + sr - u\overline{r}\ell} C^{s(\overline{r} - 1) - sr + u(\overline{r}\ell + \ell - \overline{r})}) (A^{2sa} B^{2sb} C^0)\\
        &= A^{s + 2sa - u\ell} B^{1 + r + sr + 2sb - u\overline{r}\ell} C^{s(\overline{r} - 1) - 2sa - sr + u(\ell - a)} A^{-2u\ell a + us\overline{r}(\ell - 1) + usra}.
    \end{split}
\end{equation}

By \eqref{eq.app.5} and \eqref{eq.app.6},
\begin{equation}\label{eq.imp.2}
    \begin{split}
        B^{1+r} (AB^r)^s A^{2sa} B^{2sb}
        &= (A^{s - u\ell} B^{1 + r + sr - u\overline{r}\ell} C^{s(\overline{r} - 1) - sr + u(\overline{r}\ell + \ell - \overline{r})}) (A^{2sa} B^{2sb} C^0)\\
        &\qquad A^{-ur\overline{r}\ell + us(\overline{r}\ell + \ell + 1) + usr(\overline{r} + 1) + u\varphi(s)\ell + 2us}\\
        &= A^{s + 2sa - u\ell} B^{1 + r + 2sb + sr - u\overline{r}\ell} C^{s(\overline{r} - 1) - 2sa - sr + u(\ell - a)}\\
        &\qquad A^{-2u\ell a - ur\overline{r}\ell + us(\ell - \overline{r} + 1) + u\varphi(s)\ell + usr(\overline{r} + a) + 2us}.
    \end{split}
\end{equation}

Applying \cite[Theorem 3.4]{MS3} to
\begin{align*}
    &[(AB^r)^{1+s} A^{2si} B^{2sj},B^{1+r} (AB^r)^s A^{2sa} B^{2sb}]\\
    &= [A^{1 + s + 2si} B^{r + sr + 2sj + u\overline{r}\ell} C^{-s\overline{r} - ui},
    A^{s + 2sa - u\ell} B^{1 + r + 2sb + sr - u\overline{r}\ell} C^{s(\overline{r} - 1) - 2sa - sr + u(\ell - a)}]
\end{align*}
produces
$expA\equiv u\mod 2u$,
$expB\equiv sr\ell + u(\overline{r} + 1)\mod 2u$,
$expC\equiv 1 + r + s + 2s(i + b) - sr + u(\overline{r} + i + a + 1)\mod 2u$, so
\begin{equation}\label{cong.comm.1}
    \begin{split}
        &[(AB^r)^{1+s} A^{2si} B^{2sj},B^{1+r} (AB^r)^s A^{2sa} B^{2sb}]\\
        &\equiv A^u B^{sr\ell + u(\overline{r} + 1)} C^{1 + r + s - sr + 2s(i + b) + u(\overline{r} + i+ a + 1)}\mod Z_1(J).
    \end{split}
\end{equation}

Repeated applications of \cite[Theorem 3.1]{MS3} produce\\
$[A^u,B^{r + sr + 2sj + u\overline{r}\ell}] = A^{usr}$,\\
$[C^{-s\overline{r} - ui},A^u] = 1$,\\
$[A^{1 + s + 2si},B^{sr\ell + u(\overline{r} + 1)}] = C^{sr\ell + u(\overline{r} + 1)} A^{ur\ell^2 + us\ell(\overline{r} + 1) + usr}$,\\
$[C^{-s\overline{r} - ui},B^{sr\ell + u(\overline{r} + 1)}] = 1$,\\
$[C^{sr\ell + u(\overline{r} + 1)},B^{r + sr + 2sj + u\overline{r}\ell}] = B^{usr}$,\\
$[C^{1 + r + s - sr + 2s(i + b) + u(\overline{r} + i+ a + 1)},A^{1 + s + 2si}] = A^{-2s\ell - u\ell} A^{-4u\ell(2i + b + 1) - us\ell^2 + 2us(\overline{r} + a + 1) + usr}$,\\
$[C^{1 + r + s - sr + 2s(i + b) + u(\overline{r} + i+ a + 1)},B^{r + sr + 2sj + u\overline{r}\ell}] = B^{u\ell} B^{4u\ell j + ur\ell + 2us(\overline{r} + i + j + b) + usr}$,\\
which all together yield
\begin{align*}
    &((AB^r)^{1+s} A^{2si} B^{2sj})^{[(AB^r)^{1+s} A^{2si} B^{2sj},B^{1+r} (AB^r)^s A^{2sa} B^{2sb}]}\\
    &\quad\equiv (A^{1 + s + 2si} B^{r + sr + 2sj + u\overline{r}\ell} C^{-s\overline{r} - ui})^{A^u B^{sr\ell + u(\overline{r} + 1)} C^{1 + r + s - sr + 2s(i + b) + u(\overline{r} + i+ a + 1)}}\\
    &\quad\equiv (A^{1 + s + 2si} C^{sr\ell + u(\overline{r} + 1)} B^{r + sr + 2sj + u\overline{r}\ell} C^{-s\overline{r} - ui})^{C^{1 + r + s - sr + 2s(i + b) + u(\overline{r} + i+ a + 1)}}\\
    &\quad\equiv (A^{1 + s + 2si} B^{r + sr + 2sj + u\overline{r}\ell} C^{-s\overline{r} + sr\ell + u(\overline{r} - i + 1)})^{C^{1 + r + s - sr + 2s(i + b) + u(\overline{r} + i+ a + 1)}}\\
    &\quad\equiv A^{1 + s + 2s(i + \ell) + u\ell} B^{r + sr + 2sj + u\ell(\overline{r} - 1)} C^{-s\overline{r} + sr\ell + u(\overline{r} - i + 1)}\mod Z_1(J),
\end{align*}
where the central element is $A^{4u\ell(2i + j + b + 1) + ur\ell(\overline{r} + \ell) + u\varphi(s)\ell + us(\ell - \overline{r} + 1) + usri + 2us(j + a + b + 1)}$. Then
\begin{equation}\label{eq.L.1}
    \begin{split}
        &((AB^r)^{1+s} A^{2si} B^{2sj})^{[(AB^r)^{1+s} A^{2si} B^{2sj},B^{1+r} (AB^r)^s A^{2sa} B^{2sb}]}\\
        &= A^{1 + s + 2s(i + \ell) + u\ell} B^{r + sr + 2sj + u\ell(\overline{r} - 1)} C^{-s\overline{r} + sr\ell + u(\overline{r} - i + 1)}\\
        &\quad\; A^{4u\ell(2i + j + b + 1) + ur\ell(\overline{r} + \ell) + u\varphi(s)\ell + us(\ell - \overline{r} + 1) + usri + 2us(j + a + b + 1)}.
    \end{split}
\end{equation}

Applying \cite[Theorem 3.5]{MS3} to $(A^{1 + s + 2si} B^{r + sr + 2sj + u\overline{r}\ell} C^{-s\overline{r} - ui})^{1 + 2s\ell'} = A^{expA} B^{expB} C^{expC} A^{\xi}$ produces $expA\equiv 1 + s + 2s(i + \ell') + 2u\ell'(2i + 1) + u\varphi(2s\ell') - us\ell \ell'\mod 4us$, $expB\equiv r + sr + 2sj + 4uj\ell' + u(\ell' + \overline{r}\ell) + ur\ell \ell' + us\ell'(1 - \overline{r}\ell) + 2usj\ell'\mod 4us$, $expC\equiv - s\overline{r} - sr\ell' - ui + 2uj\ell'\mod 4u$, $\xi\equiv 0\mod 4us$, so
\begin{equation}\label{eq.R.1}
    \begin{split}
        &((AB^r)^{1+s} A^{2si} B^{2sj})^{1 + 2s\ell'}\\
        &= (A^{1 + s + 2si} B^{r + sr + 2sj + u\overline{r}\ell} C^{-s\overline{r} - ui})^{1 + 2s\ell'} A^{(ur\ell(\overline{r} - 1) + u\varphi(s)\ell + us\overline{r}(\ell - 1) + usri + 2usi)(1 + 2s\ell')}\\
        &= A^{1 + s + 2s(i + \ell')} B^{r + sr + 2sj + u(\overline{r}\ell + \ell')} C^{-s\overline{r} - sr\ell' - ui}\\
        &\quad\; A^{2u\ell'(2i - 2j + 1) + ur\ell(\overline{r} - \ell' - 1) + u\varphi(s)\ell + us(\overline{r}\ell + \overline{r}\ell \ell' - \overline{r} - \ell \ell' - \ell') + 2usi + u\varphi(2s\ell') + usri}.
    \end{split}
\end{equation}

Repeated applications of \cite[Theorem 3.1]{MS3} produce\\
$[C^{-1 - r - s + sr - 2s(i + b) - u(\overline{r} + i + a + 1)},A^{s + 2sa - u\ell}] = A^{2u\ell(2a + 1) + us\ell + 2usa}$,\\
$[C^{-1 - r - s + sr - 2s(i + b) - u(\overline{r} + i + a + 1)},B^{1 + r + 2sb + sr - u\overline{r}\ell}] = B^{-2s\ell} B^{-4u\ell(i + 2b + 1) - ur\ell + us(\ell + 1) + 2usa}$,\\
$[A^{s + 2sa - u\ell},B^{-sr\ell - u(\overline{r} + 1)}] = A^{usr}$,\\
$[C^{s(\overline{r} - 1) - 2sa - sr + u(\ell - a)},B^{-sr\ell - u(\overline{r} + 1)}] = 1$,\\
$[A^{-u},B^{1 + r + 2s(b + \ell) + sr - u\overline{r}\ell}] = C^{-u} A^{-us\ell + usr}$,\\
$[C^{s(\overline{r} - 1) - 2sa - sr + u(\ell - a)},A^{-u}] = 1$,\\
which all together yield
\begin{align*}
    &(B^{1+r} (AB^r)^s A^{2sa} B^{2sb})^{[B^{1+r} (AB^r)^s A^{2sa} B^{2sb},(AB^r)^{1+s} A^{2si} B^{2sj}]}\\
    &\equiv (A^{s + 2sa - u\ell} B^{1 + r + 2sb + sr - u\overline{r}\ell} C^{s(\overline{r} - 1) - 2sa - sr + u(\ell - a)})^{C^{-1 - r - s + sr - 2s(i + b) - u(\overline{r} + i + a + 1)} B^{-sr\ell - u(\overline{r} + 1)} A^{-u}}\\
    &\equiv (A^{s + 2sa - u\ell} B^{1 + r + 2s(b + \ell) + sr - u\overline{r}\ell} C^{s(\overline{r} - 1) - 2sa - sr + u(\ell - a)})^{B^{-sr\ell - u(\overline{r} + 1)} A^{-u}}\\
    &\equiv A^{s + 2sa - u\ell} B^{1 + r + 2s(b + \ell) + sr - u\overline{r}\ell} C^{s(\overline{r} - 1) - 2sa - sr + u(\ell - a + 1)}\mod Z_1(J),
\end{align*}
where the central element is $A^{-2u\ell(2i + 3a + 4b + 3) - ur\ell(\overline{r} + 1) - us\overline{r} + u\varphi(s)\ell + usr(\overline{r} + a) + 2us}$. Then
\begin{equation}\label{eq.L.2}
    \begin{split}
        &(B^{1+r} (AB^r)^s A^{2sa} B^{2sb})^{[B^{1+r} (AB^r)^s A^{2sa} B^{2sb},(AB^r)^{1+s} A^{2si} B^{2sj}]}\\
        &= A^{s + 2sa - u\ell} B^{1 + r + 2s(b + \ell) + sr - u\overline{r}\ell} C^{s(\overline{r} - 1) - 2sa - sr + u(\ell - a + 1)}\\
        &\quad\; A^{-2u\ell(2i + 3a + 4b + 3) - ur\ell(\overline{r} + 1) - us\overline{r} + u\varphi(s)\ell + usr(\overline{r} + a) + 2us}.
    \end{split}
\end{equation}

Finally, applying \cite[Theorem 3.5]{MS3} to
\[
(A^{s + 2sa - u\ell} B^{1 + r + 2sb + sr - u\overline{r}\ell} C^{s(\overline{r} - 1) - 2sa - sr + u(\ell - a)})^{1 + 2s\ell'}
= A^{expA} B^{expB} C^{expC} A^{\xi}
\]
produces
$expA\equiv s + 2sa - u\ell + 2u\ell'(2a + 1) - us\ell \ell' + 2us\ell'(a + 1)\mod 4us$,
$expB\equiv 1 + r + 2s(b + \ell') + sr + 4ub\ell' + u(\ell' - \overline{r}\ell) + us\ell'\mod 4us$,
$expC\equiv s(\overline{r} - 1) - 2sa - sr + u(\ell - a + \ell') + 2ua\ell'\mod 4u$,
$\xi\equiv 0\mod 4us$, so
\begin{equation}\label{eq.R.2}
    \begin{split}
        &(B^{1+r} (AB^r)^s A^{2sa} B^{2sb})^{1 + 2s\ell'}\\
        &= (A^{s + 2sa - u\ell} B^{1 + r + 2sb + sr - u\overline{r}\ell} C^{s(\overline{r} - 1) - 2sa - sr + u(\ell - a)})^{1 + 2s\ell'}\\
        &\qquad A^{(-2u\ell a - ur\overline{r}\ell + us(\ell - \overline{r} + 1) + u\varphi(s)\ell + usr(\overline{r} + a) + 2us)(1 + 2s\ell')}\\
        &= A^{s + 2sa - u\ell} B^{1 + r + 2s(b + \ell') + sr + u(\ell' - \overline{r}\ell)} C^{s(\overline{r} - 1) - 2sa - sr + u(\ell - a + \ell')}\\
        &\qquad A^{2u(2a\ell' - \ell a - 2b\ell' + \ell') - ur\overline{r}\ell + us(\ell - \overline{r} + 1 - \ell \ell' - \ell') + u\varphi(s)\ell + usr(\overline{r} + a) + 2us(\ell' + 1)}.
    \end{split}
\end{equation}

This completes the calculation of both sides of \eqref{eq.LR.1} and \eqref{eq.LR.2}.

\end{document}